\documentclass[12pt, reqno]{amsart}
\usepackage{amssymb, amsthm, amsmath, amsfonts}
\usepackage{array, epsfig}
\usepackage{bbm, enumerate, textcomp}
\usepackage{color}
\usepackage{hyperref,verbatim}
\usepackage{lineno}

  \evensidemargin
\oddsidemargin
\newcommand*\patchAmsMathEnvironmentForLineno[1]{%
  \expandafter\let\csname old#1\expandafter\endcsname\csname #1\endcsname
  \expandafter\let\csname oldend#1\expandafter\endcsname\csname end#1\endcsname
  \renewenvironment{#1}%
     {\linenomath\csname old#1\endcsname}%
     {\csname oldend#1\endcsname\endlinenomath}}%
\newcommand*\patchBothAmsMathEnvironmentsForLineno[1]{%
  \patchAmsMathEnvironmentForLineno{#1}%
  \patchAmsMathEnvironmentForLineno{#1*}}%
\AtBeginDocument{%
\patchBothAmsMathEnvironmentsForLineno{equation}%
\patchBothAmsMathEnvironmentsForLineno{align}%
\patchBothAmsMathEnvironmentsForLineno{flalign}%
\patchBothAmsMathEnvironmentsForLineno{alignat}%
\patchBothAmsMathEnvironmentsForLineno{gather}%
\patchBothAmsMathEnvironmentsForLineno{multline}%
}

\begin{document}

\newcommand{\be}{\begin{equation}}
\newcommand{\ee}{\end{equation}}
\newcommand{\bea}{\begin{eqnarray}}
\newcommand{\eea}{\end{eqnarray}}
\newcommand{\beaa}{\begin{eqnarray*}}
\newcommand{\eeaa}{\end{eqnarray*}}
\newcommand{\Var}{\mathop{\mathrm{Var}}\nolimits}

\renewcommand{\proofname}{\bf Proof}
\newtheorem{conj}{Conjecture}
\newtheorem*{cor*}{Corollary}
\newtheorem{cor}{Corollary}[section]
\newtheorem{proposition}{Proposition}[section]
\newtheorem{lemma}{Lemma}[section]

\newtheorem{taggedlemmax}{Lemma}
\newenvironment{taggedlemma}[1]
 {\renewcommand\thetaggedlemmax{#1}\taggedlemmax}
 {\endtaggedlemmax}
 
\newtheorem*{crit}{Conditions for conservativity}
\newtheorem*{Hopf}{Hopf's ratio ergodic theorem}
\newtheorem{theorem}{Theorem}[section]

\theoremstyle{remark}
\newtheorem{remark}{Remark}[section]

\newfont{\zapf}{pzcmi}

\def\PP{\mathrm{P}}
\def\EE{\mathrm{E}}

\def\R{\mathbb{R}}
\def\Q{\mathbb{Q}}
\def\Z{\mathbb{Z}}
\def\N{\mathbb{N}}
\def\E{\mathbb{E}}
\def\P{\mathbb{P}}
\def\V{\mathbb{D}}
\def\ZZ{\mathcal{Z}}
\def\XX{\mathcal{X}}
\def\BB{\mathcal{B}}
\def\I{\mathbbm{1}}
\newcommand{\D}{\hbox{\zapf D}}
\newcommand{\eps}{\varepsilon}
\newcommand{\sgn}{\mathop{\mathrm{sign}}\nolimits}
\newcommand{\eqdistr}{\stackrel{d}{=}}
\newcommand{\supp}{\mathop{\mathrm{supp}}\nolimits}
\newcommand{\lin}{\mathop{\mathrm{lin}}\nolimits}
\newcommand{\Int}{\mathop{\mathrm{Int}}\nolimits}
\newcommand{\Cl}{\mathop{\mathrm{Cl}}\nolimits}
\newcommand{\Inv}{\mathop{\mathrm{Inv}}\nolimits}
\newcommand{\Sym}{\text{\rm Sym}}
\newcommand{\dist}{\operatorname{dist}}
\newcommand{\tc}{\textcolor{red}}
\renewcommand{\Re}{\operatorname{Re}}
\renewcommand{\mod}{\operatorname{mod}}
\newcommand{\Mod}[1]{\ \mathrm{mod}\ #1}

\newcommand{\overbar}[1]{\mkern 1.5mu\overline{\mkern-3mu#1\mkern-3mu}\mkern 1.5mu}
\newcommand{\todistr}{\overset{d}{\underset{n\to\infty}\longrightarrow}}
\newcommand{\toprobab}{\overset{\P}{\underset{n\to\infty}\longrightarrow}}

\title[Stationary entrance Markov chains, inducing, and level-crossings]{Stationary entrance Markov chains, \\inducing, and level-crossings of random walks}
\author{Aleksandar Mijatovi\'c} 
\address{Aleksandar Mijatovi\'c, Department of Statistcis, University of Warwick \& The Alan Turing Institute}
\email{a.mijatovic@warwick.ac.uk}

\author{Vladislav Vysotsky}
\address{Vladislav Vysotsky, University of Sussex,
Pevensey 2 Building,
Falmer Campus,
Brighton BN1 9QH,
United Kingdom
and 
St.\ Petersburg Department of Steklov Mathematical Institute,
Fontanka~27,
191011 St.\ Petersburg,
Russia}
\email{v.vysotskiy@sussex.ac.uk}


\thanks{AM is supported by the EPSRC grant EP/P003818/1 and a Fellowship at The Alan Turing Institute, sponsored by the Programme on Data-Centric Engineering funded by Lloyd's Register Foundation. VV is supported in part by the RFBR Grant 19-01-00356.}

\begin{abstract}
For a Markov chain $Y$ with values in a Polish space, consider the {\it entrance Markov chain} obtained by sampling $Y$ at the moments when it enters a fixed set $A$ from its complement $A^c$. Similarly, consider the {\it exit Markov chain}, obtained by sampling $Y$ at the exit times from $A^c$ to $A$. This paper provides a framework for analysing
invariant measures of these two types of Markov chains in the case when the
initial chain $Y$ has a known $\sigma$-finite invariant measure. Under 
certain recurrence-type assumptions ($Y$ can be {\it transient}), we give explicit formulas for invariant measures of these chains. Then we study their uniqueness and ergodicity   assuming that $Y$ is topologically recurrent, irreducible, and weak Feller. 

Our approach is based on the technique of {\it inducing} from infinite ergodic theory. This also yields, in a natural way, the versions of the results above (provided in the paper) for the classical {\it induced Markov chains}.

We give applications to random walks in $\R^d$, which we regard as
``stationary''  Markov chains started under the Lebesgue measure. We are mostly interested in dimension one, where we study the Markov chain of overshoots above the zero level of a random walk that oscillates between $-\infty$ and $+\infty$. We show that this chain is ergodic, and use this result to prove a central limit theorem for the number of level crossings for random walks with zero mean and finite variance of increments.  
\end{abstract}

\maketitle

\tableofcontents

\section{Introduction}
\subsection{Introduction and description of the main results} \label{sec:Intro} Let $S= (S_n)_{n \ge 0}$ with $S_n= S_0 + X_1+ \ldots + X_n$ be a non-degenerate random walk in $\R^d$, where $d \ge 1$, with independent identically distributed (i.i.d.) increments $X_1, X_2, \ldots$ and the starting point $S_0$ that is a random vector independent of the increments. Define the {\it state space} of the walk $S$, denoted by $\ZZ$, as the minimal topologically closed subgroup of $(\R^d, +)$ containing the topological support of the distribution of $X_1$. Assume throughout that $S_0 \in \ZZ$. Let $\lambda$ be the normalized Haar measure on $(\ZZ, +)$; in dimension one this means that $\lambda([0,x))=x$ for positive $x \in \ZZ$.

Consider the case $d=1$ and assume that either $\E X_1=0$ or $\E X_1$ does not exist. This is equivalent to assuming that the random walk $S$ {\it oscillates}, that is $\limsup S_n = -\liminf S_n = +\infty$ a.s.\ as $n \to \infty$ (see Section~\ref{Sec: Application to RWs}). In particular, $S$ can be transient. Define the {\it crossing times} of the zero level by $\mathcal{T}_0:=0$ and
\[
\mathcal{T}_{n}:= \inf\{k>\mathcal{T}_{n-1}: S_{k-1} < 0, S_k \ge 0 \text{ or } S_{k-1} \ge 0, S_k <0\}, \qquad n \in\N,
\]
and let 
\begin{equation} \label{eq: crossing chain}
\mathcal{O}_n:=S_{\mathcal{T}_n}, \qquad \mathcal{U}_n:= S_{\mathcal{T}_n - 1}, \qquad n \in \N 
\end{equation}
be the corresponding {\it overshoots} and {\it undershoots}; put $\mathcal{O}_0=\mathcal{U}_0:=S_0$. It is easy to see, using that the $\mathcal{T}_n$'s are stopping times, that the sequence $\mathcal{O}:=(\mathcal{O}_n)_{n \ge 0}$ is a Markov chain. The sequence   $\mathcal{U}:=(\mathcal{U}_n)_{n \ge 0}$ also is a Markov chain but this is far less intuitive (cf.~Lemma~\ref{lem: Markov}). 

This paper was motivated by our interest in stationarity and stability properties of the Markov chain of overshoots $\mathcal{O}$. In our companion paper~\cite{MijatovicVysotsky} we essentially showed that this chain has {\it an} invariant measure 
\begin{equation} \label{eq: pi}
\pi(dx):= \frac{c_1}{2} \bigl[ \I_{[0, \infty)}(x) \P(X_1 > x) + \I_{(-\infty,0)}(x) \P(X_1 \le x) \bigr] \lambda(dx), \qquad x \in \ZZ,
\end{equation}
where either $c_1:=2/\E|X_1|$ if $\E |X_1|<\infty$ (and $d=1$) or $c_1:=1$ otherwise. Note that $\pi$ is finite  if and only if $\E |X_1| < \infty$, in which case by the assumption of oscillation we have $\E X_1 =0$ and so $\pi$ is a probability.  

We found this measure in~\cite{MijatovicVysotsky}  computing it in a special case using an ergodic averaging argument and  assuming that an invariant measure exists. Then we proved invariance of $\pi$ in general case using quite interesting but yet complicated and apparently ad-hoc arguments based on time-reversibility. The same approach of computing or even guessing and then proving invariance was used in a number of other works (e.g.~\cite{Borovkov, Knight, PeigneWoessUnpub}, commented below) concerning stability of certain Markov chains. In all these cases this neither gives insights on the form of invariant measures nor shows how to find them. In particular, this does not explain why the Haar measure, which was introduced in ~\cite{MijatovicVysotsky} {\it only} to simplify the notation, appears in formula~\eqref{eq: pi}. Moreover, this approach does not allow one to prove uniqueness.  

A universal approach for proving uniqueness of the invariant measure is in establishing some type of distributional convergence of $\mathcal{O}_n$ towards $\pi$ as $n \to \infty$ (when $\E X_1 = 0$) starting $S$ from every $x \in \ZZ$. In this paper we set this delicate problem aside. It is considered, under additional smoothness assumptions on the distribution of $X_1$, in the companion paper~\cite{MijatovicVysotsky} using the methods that are entirely different from the ones used here. However, we were not able to establish convergence in full generality. The main difficulty is that no standard criteria of convergence apply to~$\mathcal{O}$: in particular, this chain in general is neither weak Feller (Remark~\ref{rem: not Feller}) nor $\psi$-irreducible. By the same reasoning, a priori it is unclear if the chain $\mathcal O$ has a stationary distribution regardless of moment assumptions on $S$.

This paper presents a new approach which allows us to {\it find}
invariant measures and obtain their uniqueness and ergodicity in much more
general context than level-crossings of one-dimensional random walks. 

In order to proceed to this general setting, note that the chain of overshoots has periodic structure since its values at consecutive steps have different signs. Therefore, it suffices to consider the non-negative Markov chain $O=(O_n)_{n\ge 0}$ of {\it overshoots at up-crossings} defined by $O_n:= \mathcal{O}_{2n - \I(S_0<0)}$ for $n \ge 1$ and starting at $O_0:=S_0$, and the analogous negative chain $O^\downarrow$ of {\it overshoots at down-crossings}. Similarly, define the negative-valued chain $U=(U_n)_{n \ge 0}$ of {\it undershoots at up-crossings} given by $U_n:= \mathcal{U}_{2n - \I(S_0<0)}$ for $n \ge 1$ and $ U_0:=S_0$. This chain played an important role in the proof of equality \eqref{eq: pi} in~\cite{MijatovicVysotsky}.

Observe that the
Markov chain of overshoots $O$ at up-crossings above the zero level is obtained by
sampling the one-dimensional  random walk $S$ at the moments it enters the set $[0,
\infty)$ from $(-\infty, 0)$. 
Similarly, for any Markov chain $Y$ we can consider the {\it entrance Markov chain}, denoted by $Y^{\to A}$, constructed by sampling $Y$ at the moments of  entry into an arbitrary fixed set $A$ from its complement $A^c$. We also consider the {\it exit Markov chain}, denoted by $Y^{A^c \to}$, obtained by sampling $Y$ at the exit times 
from $A^c$ to $A$; the Markov property of this sequence is not obvious and we refer to Lemma~\ref{lem: Markov} for its proof. In this notation, we have $O=S^{\to [0, \infty)}$ and $U=S^{(-\infty,0) \to}$. Note that alternating the values of the entrance and exit chains gives the {\it Markov chain of overshoots over the boundary} $\partial A$ analogous to $\mathcal O$ but we will not give it any consideration. 


We will show (Theorems~\ref{thm: duality} and \ref{thm: inducing MC}) that for any Borel set $A$ in a Polish space $\XX$ and any time-homogeneous Markov chain $Y$ that takes values in $\XX$ and has an  invariant $\sigma$-finite measure~$\mu$, the entrance chain $Y^{\to A}$ and the exit chain $Y^{A^c \to}$ have respective invariant  measures 
\begin{equation} \label{eq: measures def}
\mu_A^{entr}(dx)=\P_x(\hat{Y}_1 \in A^c) \mu(dx) \text{ on }A, \qquad 
\mu_{A^c}^{exit}(dx)= \P_x(Y_1 \in A) \mu(dx) \text{ on } A^c,
\end{equation}
where $\hat Y$ is a Markov chain {\it dual} to $Y$ with respect to $\mu$ and $\P_x(Y_0=\hat Y_0 = x)=1$, assuming that 
$Y$ visits both sets $A$ and $A^c$ infinitely often $\P_x$-a.s.\ for $\mu_{A^c}^{exit}$-a.e.~$x$, and the same holds true for $\hat Y$ (in particular, $A$ and $A^c$ are {\it recurrent sets} for $Y$ and $\hat Y$). We will also show that the chains $Y^{\to A}$ and $Y^{A^c \to}$ are recurrent and ergodic if so is $Y$ when ``started'' under $\mu$.

The assumptions above are satisfied (Remark~\ref{rem: Polish}.b) if $Y$ is recurrent started under $\mu$ and it can get from $A^c$ to $A$. We stress that this not a  requirement, that is $Y$ {\it can be  transient}.

It turns out that the exit chain $Y^{A^c \to}$ of $Y$ from $A^c$ to $A$ is  dual to the entrance chain $\hat Y^{\to A^c}$ of  $\hat Y$ into $A^c$ from $A$  with respect to the measure $\mu_{A^c}^{exit}$ (Theorem~\ref{thm: duality}). This immediately implies that  $\mu_{A^c}^{exit}$  is invariant for the exit chain  $Y^{A^c \to}$. This in turn yields invariance of $\mu_A^{entr}$ for the entrance chain $Y^{\to A}$ since formulas~\eqref{eq: measures def} are antisymmetric  in the sense that their right-hand sides interchange if we swap the chain $Y$ and the set $A$ with the dual chain $\hat Y$ and the complement set $A^c$. Therefore, we can concentrate on entrance Markov chains. 

Our  results apply for two very wide classes of Markov chains with a known invariant measure, namely {\it reversible} chains and {\it random walks on topological groups} with the Haar measure. In this paper we will only consider applications with concrete examples for random walks on $(\R^d, +)$ (see Section~\ref{Sec: Application to RWs}). Let us briefly describe possible applications for the first class of Markov chains. Recall that a chain $Y$ is reversible if it is self-dual with respect to a $\sigma$-finite measure $\mu$, which has to be invariant. For such chains, formula \eqref{eq: measures def} for $\mu_A^{entr}$ is particularly simple since we can take $\hat Y_1 =Y_1$. The assumptions can be verified as follows. A simple necessary and sufficient condition for recurrence of countable reversible Markov chains is due to Lyons~\cite{Lyons}. For transient chains other than random walks, we are aware of only one particular example with explicit characterization of all recurrent sets: by Gantert et al.~\cite[Theorem~1.7]{GPV}, a planar simple random walk conditioned on never hitting the origin visits any infinite subset of $\Z^2$ infinitely often a.s. Necessary and sufficient conditions for recurrence of a set for a countable state space Markov chain can be found in Bucy~\cite{Bucy}, Murdoch~\cite{Murdoch}, and Menshikov et al.~\cite[Theorem~2.5.8]{Menshikov+}. 

Let us stress that the above simple probabilistic explanation of invariance of the measures in \eqref{eq: measures def} neither clarifies how to {\it find} them and why they have the form given, nor allows us to prove their  uniqueness. Therefore, we use an entirely different method.

Our approach is built on {\it inducing}, a basic tool of ergodic theory, introduced by S.~Kakutani in 1943.
We need to use infinite ergodic theory since in many cases of interest the invariant measure $\mu$ of the chain $Y$ is infinite. In particular, so is  the Haar measure $\lambda$ on the subgroup $\ZZ$ of $\R^d$. Since $\lambda$  is invariant for the random walk~$S$ and the dual of $S$ with respect to $\lambda$ is $-S$, the first formula in \eqref{eq: measures def} reads as
\[
\lambda_A^{entr}(dx)=\P(X_1 \in x - A^c) \lambda(dx). 
\]
This example provides all applications of inducing presented in this paper (see Section~\ref{Sec: Application to RWs}). Our main interest is in level-crossings of random walks in dimension one, corresponding to $A=\pm [0,\infty)$. We also briefly consider other choices of $A$, two of which are discussed below.

First, let $A \subset \ZZ$ be such that both sets $A$ and $A^c$ have non-empty interior, and assume that the walk $S$ is topologically recurrent on $\ZZ$, so $d=1$ or $d=2$. Then the measure $\lambda_A^{entr}$ is always invariant for the entrance chain $S^{\to A}$ into $A$ and is finite if $A$ is bounded.

Second, $A$ is either the non-negative or the negative orthant in $\R^d$. Assume that both events $\{S_n \ge 0\}$ and $\{S_n < 0\}$ occur infinitely often a.s.; we always mean that inequalities between points in $\R^d$ hold coordinate-wise. Notice that in dimension one this assumption is equivalent to oscillation of $S$. Let us again stress that $S$ can be transient. Then the measure
\begin{equation} \label{eq: inv quadrant}
\pi_+(dx):= c_1 \I_{[0, \infty)^d}(x) (1 - \P(X_1 \le x)) \lambda(dx), \qquad x \in \ZZ,
\end{equation}
which satisfies $\pi_+=c_1 \lambda_{[0, \infty)^d}^{entr}$, is invariant for the entrance chain $S^{\to [0, \infty)^d}$ into the non-negative orthant $[0, \infty)^d$. In particular, for $d=1$ this means that $\pi_+$ is invariant for the chain $O$. Combining this with the analogous result for the chain $O^\downarrow$ of overshoots at down-crossings, which is invariant under $\pi_-:=c_1 \lambda_{(-\infty,0)^d}^{entr}$, we obtain that the measure $\pi=\frac12 \pi_+ + \frac12 \pi_-$ is invariant for the chain of overshoots $\mathcal O$. Let us mention that distributions of the same form as $\pi_+$ in $d=1$ appear on many occasions, as discussed in detail in~\cite[Sections 2.1 and~2.2]{MijatovicVysotsky}.

Our further result (Theorem~\ref{thm: inducing bijection}) implies that under the topological assumptions of recurrence, irreducibility, and weak Feller property of the chain $Y$ and, essentially, non-emptiness of the interiors of the sets $A$ and $A^c$, the questions of existence of an invariant measure, its ergodicity and uniqueness (up to a constant factor)  in the class of locally finite Borel measures have the same answer for each of the chains $Y$, $Y^{\to A}$, $Y^{A^c \to}$. In particular, since the Haar measure is known to be ergodic and unique locally finite invariant measure of the topologically recurrent random walk $S$ on $\ZZ$, this yields uniqueness and ergodicity of the invariant measure $\pi$ for the chain of overshoots $\mathcal O$. More generally (Theorem~\ref{thm: uniqueness}),
the same holds for the measure $\lambda_A^{entr}$ and the chain $Y^{\to A}$ if we assume for simplicity that $\lambda(\partial A)=0$.  Theorem~\ref{thm: inducing bijection} also is a useful tool for proving existence (Proposition~\ref{prop: existence}) of a locally finite invariant measure of the weak Feller chain $Y$ when its path empirical distributions are not tight and so the classical Bogolubov--Krylov theorem does not apply. This is based on {\it Kac-type formulas} of Proposition~\ref{prop: Kac MC}, which in a sense are inverse to those of Theorem~\ref{thm: inducing MC} obtained by inducing. {\it NB:} after this paper was finished, we found the works by Lin~\cite{Lin} and Skorokhod~\cite{Skorokhod} on invariant distributions of weak Feller topologically recurrent Markov chains, with a stronger existence result (see Proposition~\ref{prop: Skorokhod} and the preceding discussion).

Our approach also provides a natural venue to study {\it induced Markov chains} (sometimes referred to as {\it embedded} chains), obtained by sampling a Markov chain  when it belongs to an arbitrary fixed set. These chains are closely related to the entrance and exit chains, and some of our results can actually be obtained by sampling the bivariate chain $Z$ on $\XX \times \XX$, formed by the pair of two consecutive values of $Y$, when $Z$ belongs to the set $A^c \times A$. Induced chains are discussed, for example, in Revuz~\cite[Chapters~2.4 and~3.2]{Revuz}, where they appear in the context of probabilistic potential theory. For completeness of exposition, we also give versions of all our statements on existence and uniqueness of invariant measures for general induced chains. Some of these results, stated in Parts 1 of respective assertions, essentially are not new but it is hard to extract them from the literature even for recurrent chains. 

Based on our experience, the use of inducing in probability theory is mostly reverse, being related to Kac-type formulas for stationary processes and even more specifically, for Markov chains but only under the omnipresent assumption of recurrence. We stress that our main results Theorems~\ref{thm: duality},~\ref{thm: inducing MC} and Proposition~\ref{prop: Kac MC} also apply to {\it transient} Markov chains. 

We are not aware of any applications of (direct) inducing in problems specifically related to random
walks. In the context of level-crossings by one-dimensional random walks, the
classical and universal  tool is the Wiener--Hopf factorization.
Among the vast literature on the topic,  the works of Baxter~\cite{Baxter},
Borovkov~\cite{Borovkov}, and Kemperman~\cite{Kemperman} are  the most
relevant  to the questions considered here. These papers rely  on the
Wiener--Hopf factorization (which does not yield much for our problem!), in contrast to our entirely different approach.
As far as we know, fluctuation theory for random walks in $\R^d$ is limited to a single paper by Greenwood and Shaked~\cite{GreenwoodShaked}, and our formulas for invariant measures such as \eqref{eq: inv quadrant} are the only explicit high-dimensional results available.
The idea to study properties of random walks regarding them as ``stationary'' processes starting from the Haar measure is novel to us, and we have never seen it in the vast literature on random walks. 

We are not aware of any works concerning entrance and exit Markov chains of any type. 

We conclude the paper with a number of one-dimensional results on level-crossings of  random walks presented in Section~\ref{Sec: L_n}. The main result there is the limit theorem for the number of level-crossings motivated below in Section~\ref{Sec: motivation}. Crucially, this theorem requires no assumption other than that the increments are zero-mean and have finite variance. We also present several formulas for expected occupation times between level-crossings. 

\subsection{Motivation and related questions} \label{Sec: motivation} This section concerns only the one-dimensional case. There are several good reasons to study overshoots of random walks besides purely theoretical interest. First, there is connection to the local time. Namely, the Markov chain $\mathcal{O}$ features in Perkins's~\cite{Perkins} definition of the {\it local time of a random walk}. There is no conventional definition of this notion, see Cs\"org\H o and R\'ev\'esz~\cite{CsorgoRevesz} for other versions. Let
\begin{equation} \label{eq: Ln def}
L_n:=\max\{k \ge 0 : \mathcal{T}_k \le n\}
\end{equation}
be the {\it number of zero-level crossings} of the walk by time $n$, and let $\ell_0$ be the local time at level $0$ at time $1$ of a standard Brownian motion. Perkins~\cite[Theorem~1.3]{Perkins} proved that for a zero
mean random walk $S$ with finite variance $\sigma^2:= \E X_1^2$ and starting at $S_0=0$, one has
\begin{equation} \label{eq: Perkins 0}
\frac{1}{\sqrt{n}} \sum_{k=1}^{L_n} |\mathcal{O}_k| \todistr \sigma \ell_0.
\end{equation}
To the best of our knowledge, all the other limit theorems for the local time (under either definition) of a random walk with finite variance require additional smoothness assumptions on the distribution of increments. 

Under the above assumptions on the random walk, by ergodicity of the Markov chain $\mathcal{O}$ (Theorem~\ref{thm: uniqueness}), we have $\frac1n \sum_{k=1}^n |\mathcal{O}_k| \to \int_\ZZ |x| \pi(dx)$ a.s.\ for $\pi$-a.e.\ starting point $S_0=x\in \ZZ$; we will show that this convergence actually holds for every $x$. Hence \eqref{eq: Perkins 0} immediately gives a limit theorem for the number of level crossings $L_n$ divided by $\sqrt n$ (Theorem~\ref{thm: level-crossings}). Under the optimal moment assumption $\E X_1^2 < \infty$,  limit theorems of such type  were first obtained  in the early 1980s by  A.N.\ Borodin, who studied more general questions of convergence of additive functionals of consecutive steps of random walks; see Borodin and Ibragimov~\cite[Chapter V]{BorodinIbragimov} and references therein. However, Borodin's method assumes that the distribution of increments of the walk is either aperiodic integer-valued or has a square-integrable characteristic function, and hence it is (Kawata~\cite[Theorem~11.6.1]{Kawata}) absolutely continuous. We stress that our result, Theorem~\ref{thm: level-crossings}, does not require any smoothness assumptions. 

Second, the Markov chain $O$ appeared in the study of the probabilities that the {\it integrated random walk} $(S_1 + \ldots + S_k)_{1 \le k \le n}$ stays positive; see Vysotsky~\cite{Vysotsky2010, Vysotsky2014}. The main idea of the approach of~\cite{Vysotsky2010, Vysotsky2014} is in a) splitting the trajectory of the walk into consecutive ``cycles'' between the up-crossing times; and b)  using that for certain particular distributions of increments, e.g.\ in the case when the distribution $\P(X_1 \in \cdot \, | X_1>0)$ is exponential, the overshoots $(O_n)_{n \ge 1}$ are stationary (actually, i.i.d.) regardless of the starting point $S_0$. The current paper was originally motivated by the question whether this approach can be extended to general distributions of increments. 

Third, there is a close connection to so-called {\it switching random walks}.  Define the {\it switching ladder times}
\[
\mathcal{T}_0':=0, \quad \mathcal{T}_n':= 
\begin{cases}
\inf \{k>\mathcal{T}_{n-1}': S_k \le S_{\mathcal{T}_{n-1}'} \} , &  \text{if } S_{\mathcal{T}_{n-1}'} \ge 0,\\
\inf \{k>\mathcal{T}_{n-1}': S_k \ge S_{\mathcal{T}_{n-1}'} \} , &  \text{if } S_{\mathcal{T}_{n-1}'} < 0,\\
\end{cases}
\qquad n \in \N.
\]
and the {\it switching ladder heights} $Z_n:= S_{\mathcal{T}_n'}$, $n \ge 0$. The random sequence  $Z=(Z_n)_{n \ge 0}$ belongs to a special type of Markov chains that we call {\it random walks with switch} at zero, whose distributions of increments depend only on the sign of the current position of the chain; the other authors call them  {\it oscillating random walks} but this can be easily confused with oscillation in the sense used in Section~\ref{sec:Intro}. More precisely, the transition probabilities $P(x, dy)$ of such a chain $Y$ are of the form $P(x, dy) = P_{\sgn x}(dy -x)$ for $x \neq 0$ and $P(0, dy) = \alpha P_+(dy) + (1-\alpha) P_-(dy)$, where $P_+$ and $P_-$ are two probability distributions on  $\R$ and $\alpha \in [0,1]$. In the symmetric case when $P_+(dy)=P_-(-dy)$ and $\alpha=\frac12$, the sequence $(|Y_n|)_{n \ge 0}$ is called a {\it reflected random walk}. 

Random walks with switch were introduced by Kemperman~\cite{Kemperman}, and then considered in a few works including that by Borovkov~\cite{Borovkov}. Reflected random walks received much more attention; see Peign\'e and Woess~\cite{PeigneWoess} for  the most recent and comprehensive list of references and generalizations to processes of iterated i.i.d.\ random continuous mappings. Their relevance to the present paper is that the overshoots above zero level of the switching ladder heights chain $Z$ coincide with those of the random walk~$S$. 
Recently we learned about the unpublished work by Peign\'e and Woess~\cite{PeigneWoessUnpub} who found the invariant measure for a reflected random walk $(|Y_n|)_{n \ge 0}$ sampled at the moments of reflection at zero. In our terminology, this is the random sequence of absolute values of {\it non-strict} overshoots above the zero level by $Y$. One can check using the Wiener--Hopf factorization that in the special case when $Y$ is the switching ladder heights chain $Z$ generated by a random walk $S$  on $\ZZ=\R$ with symmetrically distributed increments, the invariant measure of~\cite[Theorem 4.7]{PeigneWoessUnpub} equals $\pi(|\cdot|)$. The other notable fact is that, if $\E X_1 =0$ and $\E X_1^2 <\infty$, then $\pi$ is an invariant distribution for a random walk with switch at zero defined by $P_+=\P(X_1 \in \cdot | X_1 <0)$, $P_-=\P(X_1 \in \cdot | X_1 >0)$, and $\alpha =1$; this can be shown using a stationary distribution for $Y$ found in~\cite{Borovkov}. We will use our inducing approach to explore these connections and show further relations to classical stationary distributions of renewal theory in the separate paper~\cite{Vysotsky2018}.  

\subsection{Structure of the paper} 
In Section~\ref{Sec: duality} we carefully define the entrance and exit chains and use probabilistic arguments to check their Markov property and the crucial property of duality. In Section~\ref{sec: stationarity MC} we study stationarity of induced, entrance and exit chains using the idea of inducing from ergodic theory: in Section~\ref{sec: setup MCs} we provide a self-contained setup needed to apply inducing in the context of Markov chains; in Section~\ref{sec: Inducing MCs} we show its use to find invariant measures for the three types of chains in question, obtained from a general Markov chain on a topological space; and in Section~\ref{sec: MC e and u} we study existence and uniqueness of these invariant measures for a more specific class  of weak Feller chains on Polish spaces. In Section~\ref{Sec: Application to RWs} we give applications to the entrance chains sampled from random walks in arbitrary dimension, including the chains of overshoots above a level in dimension one. In Section~\ref{Sec: L_n} present further one-dimensional results on level-crossings of random walks, including the limit theorem for the number of level-crossings. We conclude the paper with the Appendix, where we give the basics on the use of inducing in infinite ergodic theory. We present them in a nearly self-contained way and with proofs to ensure assumptions more general than those we found in the literature.

\section{Entrance and exit Markov chains and their duality} \label{Sec: duality}

\subsection{Basic notation} Throughout this paper $\XX$ will be a topological space. We always equip subsets of $\XX$ with the subspace (induced) topology. All the measures on $\XX$ (and its subsets) considered here will be {\it Borel measures}, that is defined on the Borel $\sigma$-algebra $\mathcal{B}(\XX)$.  

Let $Y=(Y_n)_{n \ge 0}$ be a time-homogeneous Markov chain taking values in $\XX$. By saying this, we assume that $Y$ is defined on some generic probability space $(\Omega, \mathcal F, \P)$ and $Y$ has a transition kernel $P$ on $\XX$ under $\P$. We also assume that $(\Omega, \mathcal F)$ is equipped with a family of probability measures $\{\P_x \}_{x \in \XX}$ such that: $Y$ is a Markov chain with the transition kernel $P$ under $\P_x$ and $\P_x(Y_0=x) =1$ for every $x \in \XX$, and the function $x \mapsto \P_x(Y \in B)$ is measurable for any set $B \in \mathcal{B}(\XX)^{\otimes \N_0}$, where $\N_0:=\N \cup \{0\}$. By the Ionescu Tulcea extension theorem (Kallenber~\cite[Theorem~6.17]{Kallenberg}), such family of measures always exists for any transition kernel on $\XX$. For any Borel $\sigma$-finite measure $\mu$ on $\XX$, denote $\P_\mu := \int_\XX \P_x(\cdot) \mu(dx)$. Then $Y_0$ has  ``distribution'' $\mu$   under $\P_\mu$, in which case we say that $Y$ {\it starts under} $\mu$. Although $\mu$ is not necessarily a probability, we prefer to (ab)use probabilistic notation and terminology as above,  instead of using respective notions of general measure theory. Denote by $\E_x$ and $\E_\mu$ respective expectations (Lebesgue integrals) over $\P_x$ and~$\P_\mu$. 

We say that a transition kernel $\hat P$ on $\XX$ is {\it dual} to $P$ with respect to a $\sigma$-finite measure $\mu$ on $\XX$ if the following equality of measures on $\BB(\XX) \otimes \BB(\XX)$ is satisfied:
\begin{equation} \label{eq: detailed bal}
\mu(dx) P(x, dy) = \mu(dy) \hat P(y, dx), \qquad x, y \in \XX.
\end{equation}
This equality is called the {\it detailed balance condition}. We will also use this definition of duality for Markov chains on subsets of $\XX$ of full measure $\mu$. If $\XX$ is separable, there cannot be two distinct$\Mod{\mu}$ kernels dual to $P$ (see Revuz~\cite[Lemma~4.7 in Chapter~2]{Revuz}); if $\XX$ is a Polish space, then a dual kernel  always exists and is unique$\Mod{\mu}$ (see Section~\ref{sec: setup MCs}). Two Markov chains $Y$ and  $\hat Y$ on $\XX$ are {\it dual} if so are their transition operators. In other words, assuming for convenience that $\hat Y_0=Y_0$, we have the {\it time-reversal  equality} 
\begin{equation*} 
\P_\mu((Y_0, Y_1) \in B) = \P_\mu((\hat Y_1, \hat Y_0) \in B), \qquad B \in \BB(\XX) \otimes \BB(\XX).
\end{equation*}
Note that this equality implies that  the measure $\mu$ is {\it invariant}  both chains $Y$ and $\hat Y$, that is $\P_\mu(Y_1 \in \cdot)=\P_\mu(\hat Y_1 \in \cdot) = \mu$. 
More generally, for any $k \ge 1$ we have
\begin{equation} \label{eq: time-reversal}
\P_\mu((Y_0, \ldots, Y_k) \in B) = \P_\mu((\hat Y_k, \ldots,  \hat Y_0) \in B), \qquad B \in \BB(\XX)^{\otimes (k+1)}.
\end{equation}

\subsection{Entrance and exit Markov chains} \label{sec: entr exit} Fix a Borel set $A \subset \XX$. Define the {\it entrance times to $A$ from $A^c$} of the chain $Y$ by $T_0^{\to A}:=0$ and
\begin{multline*} 
\shoveleft{T_{n}^{\to A}:= \inf\{k>T_{n-1}^{\to A}: Y_{k-1}\in A^c, Y_k \in A \} \quad \text{ on } \{Y_k \in A \text{ i.o.}, Y_k \in A^c \text{ i.o.}\}, \, n \in \N,}
\end{multline*}
and the sequences of {\it entrances to $A$ from} $A^c$ and  {\it exits from $A^c$ to $A$} for $Y$, respectively, by
\begin{multline*} 
\shoveleft{Y_n^{\to A}:=Y_{T_{n}^{\to A}} \quad \text{and} \quad Y_n^{A^c \to}:=Y_{T_{n}^{\to A}-1} \quad \text{ on } \{Y_k \in A \text{ i.o.}, Y_k \in A^c \text{ i.o.}\}, \, n \in \N,}
\end{multline*}
and $Y_0^{\to A}=Y_0^{A^c \to}:=Y_0$, where ``i.o.'' stands for ``infinitely often''. 

We will need to clarify for which initial values $Y_0$ of the Markov chain $Y$ the i.o.\ events occur. To this end, we define the set $N_{A', Y}$ (in short, $N_{A'}$) as follows:
\[
N_{A', Y}:= \{x \in \XX: \P_x(Y_k \in A' \text{ i.o.})=1\}, \qquad A' \in \BB(\XX).
\]
Conditions when the sets $N_A$ and $N_A \cap N_{A^c}$ are non-empty will be discussed in Section~\ref{sec: stationarity MC}. Define the {\it exit set of $Y$ from $A^c$}, denoted by $(A^c)_{ex, Y}$ (in short, $A^c_{ex}$):
\begin{equation} \label{eq: Aex}
A^c_{ex, Y}:=\{x \in A^c: P(x,A)>0\}.
\end{equation}

The following result justifies referring to the sequences $Y^{\to A}$ and $Y^{A^c \to}$ as {\it entrance} and {\it exit} Markov chains. 

\begin{lemma} \label{lem: Markov}
Let $Y$ be a Markov chain taking values in a topological space $\XX$. Let $A  \subset \XX $ be a Borel set such that $N_A \cap N_{A^c} $ is non-empty. 
Then the entrance sequence $Y^{\to A}$ and the exit sequence $Y^{A^c \to}$ are Markov chains on the respective sets $A \cap N_A \cap N_{A^c}$ and $A^c_{ex} \cap N_A \cap N_{A^c}$, with the probability transition kernels
\[
P_A^{entr}(x, dy):= \P_x(Y_1^{\to A} \in dy), \qquad x,y \in A \cap N_A \cap N_{A^c}
\]
and 
\begin{multline} \label{eq: exit chain}
\shoveleft{P_{A^c}^{exit}(x, dy):= \int_A \P_z(Y_1^{A^c \to} \in dy)  \P_x(Y_1 \in dz|Y_1 \in A), \qquad x,y \in A^c_{ex} \cap N_A \cap N_{A^c}.}
\end{multline}
\end{lemma}


We say that a Borel measure $\nu$ on $\XX$ is {\it proper}  for the  chains $Y^{\to A}$ and $Y^{A^c \to}$ if $\nu$ is supported on $N_A \cap N_{A^c}$, i.e.\ $\nu((N_A \cap N_{A^c})^c) = 0$. We will start these chains only under their proper measures. 

\begin{proof} 
We will use that for any $A' \in \BB(\XX)$, the set $N_{A'}$ is absorbing for $Y$, in the sense that
\begin{equation} \label{eq: NA}
\P_x(Y_1 \in N_{A'})= 1, \qquad x \in N_{A'}.
\end{equation}
Indeed,  for every $x \in N_{A'}$ we have
\[
1= \P_x( Y_k \in A' \text{ i.o.})=\int_\XX \P_x(Y_k \in A' \text{ i.o.} | Y_1 =y) \P_x(Y_1 \in dy) =  \int_\XX \P_y(Y_k \in A' \text{ i.o.} ) P(x, dy). 
\]
Hence $\P_y(Y_k \in A' \text{ i.o.} )=1$, i.e.\ $y \in N_A'$, for $P(x, \cdot)$-a.e.\ $y$. That is, $\P_x(Y_1 \in N_{A'})=1$.


It follows from equality \eqref{eq: NA} that the non-empty set $N_A \cap N_{A^c}$ is absorbing for $Y$. Then clearly $A \cap N_A \cap N_{A^c}$ and $A^c_{ex} \cap N_A \cap N_{A^c}$ are non-empty. Hence for every $x \in N_A \cap N_{A^c}$, the sequences $(Y^{A^c \to}_n)_{n \ge 1}$ and $(Y^{\to A}_n)_{n \ge 1}$ belong respectively to $A^c_{ex} \cap N_A \cap N_{A^c}$ and  $A \cap N_A \cap N_{A^c}$, and thus have infinitely many terms, $\P_x$-a.s. Then the sequence $(T_n^{\to A})_{n \ge 1}$ of entrance times to $A$ from $A^c$ is infinite $\P_x$-a.s. This is a sequence of stopping times with respect to $Y$.  Then it follows from the equality $Y^{\to A}=(Y_{T_{n}^{\to A}})_{n \ge 0} $ that $Y^{\to A}$ is a Markov chain under~$\P_x$. We omit the proof of this standard fact, which is based on the strong Markov property of $Y$. The formula for the transition kernel $P_A^{entr}$of $Y^{\to A}$ is evident; $P_A^{entr}$ is a probability kernel on $A \cap N_A \cap N_{A^c}$ by \eqref{eq: NA}.

The Markov property of the exit sequence $Y^{A^c \to}$ is not evident since  $(T_n^{\to A}-1)_{n \in \N}$ are not stopping times with respect to $Y$. We see from definition~\eqref{eq: exit chain} that $P_{A^c}^{exit}$ (in short, $P_{A^c}^{ex}$) is a probability kernel on $A^c_{ex} \cap N_A \cap N_{A^c}$ by \eqref{eq: NA}. To prove that $Y^{A^c \to}$ is a Markov chain on $A^c_{ex} \cap N_A \cap N_{A^c}$ with the transition kernel $P_{A^c}^{ex}$, it suffices to show that for any $x \in N_A \cap N_{A^c}$, integer $n \ge 2$, and Borel sets $B_1, B_2, \ldots \subset A^c_{ex} \cap N_A \cap N_{A^c}$, we have
\[
\P_x(Y_1^{A^c \to} \in B_1, \ldots, Y_n^{A^c \to} \in B_n) = \int_{B_1} \P_x(Y_1^{A^c \to} \in dx_1) \int_{B_2} P_{A^c}^{ex}(x_1, d x_2) \ldots  \int_{B_n} P_{A^c}^{ex}(x_{n-1}, d x_n). 
\]

The proof is by induction. Let $n=2$. Since $T_1^{\to A}$ is finite $\P_x$-a.s., it is true that
\begin{align*}
\P_x(Y_1^{A^c \to} \in B_1, Y_2^{A^c \to} \in B_2) &=\sum_{k=1}^\infty \P_x(T_1^{\to A}=k, Y_{k-1} \in B_1, Y_2^{A^c \to} \in B_2) \\
&= \sum_{k=1}^\infty \P_x\big(T_1^{\to A}>k-1, Y_{k-1} \in B_1, Y_k \in A, Y_2^{A^c \to} \in B_2\big) \\
&=\sum_{k=1}^\infty \int_{B_1} \P_x(T_1^{\to A}>k-1, Y_{k-1} \in d x_1) \\ 
&\qquad \qquad \times \P_x \big( Y_k \in A, Y_2^{A^c \to} \in B_2\big| \big.Y_{k-1} = x_1, T_1^{\to A}>k-1 \big).
\end{align*}
By the Markov property of $Y$, for $\P_x(Y_{k-1} \in \cdot)$-a.e.\ $x_1 \in B_1$ and every $k \ge 1$ we have
\begin{align*}
\P_x\big( Y_k \in A, Y_2^{A^c \to} \in B_2\big| \big.Y_{k-1} = x_1, T_1^{\to A}>k-1 \big) &= \P_{x_1}( Y_1 \in A, Y_2^{A^c \to} \in B_2) \\
&= \int_A \P_z (Y_1^{A^c \to} \in B_2) \P_{x_1}(Y_1 \in d z). 
\end{align*}

On the other hand, from definition~\eqref{eq: exit chain} of  $P_{A^c}^{ex}$ we see that
\begin{equation} \label{eq: P1}
\int_A \P_z (Y_1^{A^c \to} \in B_2) \P_{x_1}(Y_1 \in d z) = \P_{x_1}(Y_1 \in A) P_{A^c}^{ex}(x_1, B_2).
\end{equation}
Putting everything together, we obtain
\begin{align}
\P_x(Y_1^{A^c \to} \in B_1, Y_2^{A^c \to} \in B_2) 
&=\sum_{k=1}^\infty \int_{B_1}  \P_x(T_1^{\to A}>k-1, Y_{k-1} \in d x_1) \P_{x_1}(Y_1 \in A) P(x_1, B_2) \notag\\
&=\int_{B_1} \P_x(Y_1^{A^c \to} \in dx_1) \int_{B_2} P_{A^c}^{ex}(x_1, d x_2), \label{eq: Markov exit} 
\end{align}
as required in the case $n=2$. This proves the basis of induction.

To prove the inductive step, we proceed exactly as above and arrive at
\begin{align*}
&\mathrel{\phantom{=}}   \P_x(Y_1^{A^c \to} \in B_1, \ldots, Y_{n+1}^{A^c \to} \in B_{n+1}) \\
 &=\sum_{k=1}^\infty \int_{B_1} \P_x(T_1^{\to A}>k-1, Y_{k-1} \in d x_1) \int_A \P_z(Y_1^{A^c \to} \in B_2, \ldots, Y_n^{A^c \to} \in B_{n+1}) \,  \P_{x_1}(Y_1 \in d z).
\end{align*} 
We now use the assumption of induction for the probability under $\int_A$ when $z \in N_A \cap N_{A^c}$. Since $\P_{x_1}(Y_1 \in N_A \cap N_{A^c}) =1$ for every $x_1 \in  N_A \cap N_{A^c}$ because the set $N_A \cap N_{A^c}$ is absorbing for $Y$, we get
\[
\int_A \P_z(Y_1^{A^c \to} \in B_2, \ldots, Y_n^{A^c \to} \in B_{n+1}) \, \P_{x_1}(Y_1 \in d z) = \int_A \P_{x_1}(Y_1 \in d z) \int_{B_2} f(x_2) \P_z(Y_1^{A^c \to} \in d x_2), 
\]
where $f$ is a non-negative measurable function on $B_2$ given by
\[
f(x_2):= \int_{B_3} P_{A^c}^{ex}(x_2, d x_3) \ldots  \int_{B_{n+1}} P_{A^c}^{ex}(x_n, d x_{n+1}). 
\]

We claim that  for any $x_1 \in B_1$ and any non-negative measurable function $g$ on $B_2$, 
\begin{equation} \label{eq: Pf}
\int_A \P_{x_1}(Y_1 \in d z) \int_{B_2} g(x_2) \P_z(Y_1^{A^c \to} \in d x_2)  = \P_{x_1}(Y_1 \in A) \int_{B_2} g(x_2) P_{A^c}^{ex}(x_1, d x_2).
\end{equation}
Indeed, for indicator functions $g$ this holds by definition~\eqref{eq: exit chain} of  $P_{A^c}^{ex}$; cf.~\eqref{eq: P1}. Hence, \eqref{eq: Pf} holds for simple functions (i.e.\ finite linear combinations of indicator functions) by additivity of the three integrals in~\eqref{eq: Pf}. Finally, since any non-negative measurable function $g$ can be represented as pointwise limit of a pointwise non-decreasing sequence of simple functions,  equality~\eqref{eq: Pf} follows from the monotone convergence theorem.

Putting everything together and applying \eqref{eq: Pf} with $g=f$ establishes the inductive step (exactly as we obtained \eqref{eq: Markov exit} applying \eqref{eq: P1} in the case $n=2$).
\end{proof}

\subsection{Duality of entrance and exit Markov chains} \label{sec: duality}

The following assertion is crucial.

\begin{theorem} \label{thm: duality}
Let $Y$ and $\hat Y$ be Markov chains that take values in a topological space $\XX$, are dual with respect to a $\sigma$-finite invariant Borel measure $\mu$, and satisfy $Y_0=\hat Y_0$. Let $A \in \BB(\XX)$ be a set such that $N:= N_{A, Y} \cap N_{A^c, Y} \cap N_{A, \hat Y} \cap N_{A^c, \hat Y}$ is non-empty. Then the exit chain $Y^{A^c \to}$ and the entrance chain $\hat Y^{\to A^c}$ are dual with respect to the Borel measure on $A^c$ given~by
\begin{equation*} 
\tilde \mu_{A^c, Y}^{exit}(dx):= \I_N(x) \P_x(Y_1 \in A) \mu(dx), \qquad  x \in A^c.
\end{equation*}
\end{theorem}

Recall that there always exists a chain $\hat Y$ dual to $Y$ if $\XX$ is a Polish space (see Section~\ref{sec: setup MCs}).

Note in passing that in the special case when $Y$ is an oscillating  random walk $S$ on $\ZZ=\R$ and $A=[0,\infty)$, we actually have (see \cite[Section~2.4]{MijatovicVysotsky}) a quite surprising representation of the transition kernels of the chains $S^{\to A}$ and $-S^{A^c \to}$ (i.e.\ $O$ and $-U$) as products of two kernels {\it reversible} (self-dual) with respect to $\lambda_{A, S}^{entr}=\pi_+$; see the Introduction for the notation.

\begin{cor} \label{cor: dual}
The  entrance chain $Y^{\to A}$ and the exit chain ${\hat Y}^{A \to}$ are dual with respect to the measure $\tilde \mu_{A, \hat Y}^{exit}$ on $A$.
\end{cor}

This follows if we apply Theorem~\ref{thm: duality} to the chain $\hat Y$ and the set $A^c$ instead of $Y$ and~$A$.

\begin{cor} \label{cor: invar tilde}
The measure $\tilde \mu_{A^c, Y}^{exit}$ is invariant for the exit chain $Y^{A^c \to}$ and the measure $\tilde \mu_{A, \hat Y}^{exit}$  is invariant for the entrance chain $Y^{\to A}$. 
\end{cor}

The follows immediately from the dualities in Theorem~\ref{thm: duality}  and Corollary~\ref{cor: dual}.

\begin{cor} \label{cor: invar}
The measures $\mu_{A^c}^{exit} $ and $ \mu_A^{entr}$ (defined in  \eqref{eq: measures def}) are invariant for the respective chains $Y^{A^c \to}$ and $Y^{\to A}$ when $\mu_{A^c}^{exit}(N^c)=\mu_A^{entr}(N^c)=0$ (e.g., this holds if $\mu(N^c)=0$).
\end{cor}

This follows from Corollary~\ref{cor: invar tilde} since the assumption yields  $\tilde \mu_{A^c, Y}^{exit} = \mu_{A^c}^{exit} $ and $\tilde \mu_{A, \hat Y}^{exit} =\mu_A^{entr}$. These inequalities hold if $\mu(N^c)=0$ since $\mu_{A^c}^{exit}  \ll \mu$ and  $\mu_{A}^{entr}  \ll \mu$. 
Note that in general, they may cease to be true, for example, if the set $N$ is empty.

\begin{proof}[{\bf Proof of Theorem~\ref{thm: duality}}]
By Lemma~\ref{lem: Markov}, the sequences  $Y^{A^c \to}$ and $\hat Y^{\to A^c}$  are Markov chains on the respective sets $A^c_{ex, Y} \cap N$ and $A^c \cap N$. Since the measure $\tilde \mu_{A^c, Y}^{exit}$ is supported on $A^c_{ex, Y} \cap N$, the duality stated will follow once we check the detailed balance condition
\[
\tilde \mu_{A^c, Y}^{exit}(dx) P_{A^c}^{exit}(x, dy) = \tilde \mu_{A^c, Y}^{exit}(dy) \hat P_{A^c}^{entr}(y, dx), \qquad x, y \in A^c_{ex, Y} \cap N.
\]
By definition of $\tilde \mu_{A^c, Y}^{exit}$, this amounts to showing that for any Borel sets $B_1, B_2 \subset A^c_{ex, Y} \cap N$, 
\begin{equation} \label{eq: main duality}
\int_{B_1} P_{A^c}^{exit}(x, B_2)  \P_x( Y_1 \in A) \mu(dx) = \int_{B_2} \hat P_{A^c}^{entr}(y, B_1)  \P_y( Y_1 \in A) \mu(dy). 
\end{equation}

Let us transform the l.h.s. The chain  $\hat Y$ starts under $\mu$ since $\hat Y_0 = Y_0$ by assumption. Then, using at first definition \eqref{eq: exit chain} of the transition operator $P_{A^c}^{exit}$, we get
\begin{align*}
\text{LHS } \eqref{eq: main duality} &=\int_{B_1} \mu(dx) \int_A \P_z(Y_1^{A^c \to} \in B_2)  \P_x(Y_1 \in dz)  \\
&= \sum_{k,m=1}^\infty \P_\mu \big((Y_n)_{n=0}^{k+m+1} \in B_1 \times (A \cap N)^k \times (A^c \cap N)^{m-1} \times B_2 \times (A \cap N) \big)\\
&= \sum_{k,m=1}^\infty \P_\mu \big((\hat Y_n)_{n=0}^{k+m+1} \in 
(A \cap N) \times B_2 \times (A^c \cap N)^{m-1} \times (A \cap N)^k \times B_1 \big),
\end{align*}
where the third equality followed from duality of $Y$ and $\hat Y$ by~\eqref{eq: time-reversal}, and in the second equality we used that the set $N$ is absorbing for the chain $Y$ by \eqref{eq: NA}, the assumption $B_1 \subset N$, and the fact that $T_2^{\to A}$ is finite $\P_x$-a.s.\ for every $x \in N$ by $N \subset N_{A,Y} \cap N_{A^c, Y}$. We transform the last sum above using the same reasoning applied to the dual chain $\hat Y$ instead of $Y$:
\begin{align*}
\text{LHS } \eqref{eq: main duality}  &= \E_\mu[\I_{A \cap N}(\hat Y_0) \I_{B_2}(\hat Y_1) \hat P_{A^c}^{entr}(\hat Y_1, B_1)] \\
&=  \E_\mu[\I_{A \cap N}(Y_1) \I_{B_2}(Y_0) \hat P_{A^c}^{entr}(Y_0, B_1)]  = \text{RHS } \eqref{eq: main duality}, 
\end{align*}
where in the second equality we again used duality of $Y$ and $\hat Y$. This finishes the proof.
\end{proof}

\begin{remark} \label{rem: duality} 
Assuming that the transition kernel $P$ of the chain $Y$ admits a dual kernel $\hat P$ with respect to $\mu$, define the  {\it entrance set into $A$} of $Y$ started under $\mu$:
\begin{equation} \label{eq: Aen}
A_{en, Y}(\mu) :=  \{x \in A: \hat P(x, A^c) >0\}.
\end{equation}
It is readily seen  from the time-reversal equality (see~\eqref{eq: time-reversal}  for $k=1$) that this set supports the entrance measure  $\mu_A^{entr}$, given by $\mu_A^{entr} = \P_\mu(Y_0 \in A^c, Y_1 \in A \cap \cdot)$, in the sense that $\mu_A^{entr} (A\setminus A_{en, Y}(\mu))=0$ and $\mu_A^{entr}(A')>0$ for every Borel set $A' \subset A_{en, Y}(\mu) $ such that $\mu(A')>0$. This ensures that $A_{en, Y}(\mu) $ does not depend$\Mod{\mu}$ on the dual kernel $\hat P$ chosen.

We now can simplify the assumptions of Corollary~\ref{cor: invar}  using the  equivalences
\begin{equation*}
\mu_{A^c}^{exit}(N^c)=0 \Leftrightarrow \mu_A^{entr}(N^c)=0 \Leftrightarrow 
\Big( \mu(A^c_{ex} \cap N_{A_{en}, Y'}) = \mu(A_{en} \cap N_{A^c_{ex}, Y'}) = 0 \text{ for } Y' \in \{Y, \hat Y\} \Big),
\end{equation*}
where $A_{en}=A_{en,Y}(\mu)$ and $A^c_{ex}=(A^c)_{ex,Y}$. These equivalences follow rather straightforwardly and we will not prove them in full. The fact that the third condition implies the first two is essentially shown in the proof of Part 2b of Theorem~\ref{thm: inducing MC} below.

Notice that the third condition is symmetric in the sense that it does no change if we substitute $(Y,A)$ by $(\hat Y, A^c)$. This simply interchanges the sets $A_{en}$ and $A^c_{ex}$ since by \eqref{eq: Aen},
\begin{equation} \label{eq: AenAex symm}
A_{en, Y}(\mu) =A_{ex, \hat Y} \Mod{\mu} \qquad \text{and} \qquad  A^c_{ex, Y} =A^c_{en, \hat Y}(\mu) \Mod{\mu}.
\end{equation}
\end{remark}

\section{Entrance and exit Markov chains via the method of inducing} \label{sec: stationarity MC}

In this section we study stationarity of general entrance and exit Markov chains using the results of infinite ergodic theory. The application of these results is rather straightforward for recurrent Markov chains and amounts to working with measure preserving shifts on sequences. For transient chains, we need to introduce additional construction allowing the time to run backwards in order to let us work with {\it invertible} two-sided shifts on bilateral sequences. The simple time-reversal argument used in the proof of Theorem~\ref{thm: duality} suggests that this extension is indeed natural.

For completeness of exposition, we give all results on entrance and exit chains together with analogous statements for closely related {\it induced Markov chains}. This addition is very natural within the context used, and it does not take a lot of effort to provide it. The corresponding results (presented in Parts~1 of Theorems~\ref{thm: inducing MC},~\ref{thm: inducing bijection} and Proposition~\ref{prop: Kac MC})  are not new, morally or essentially, but it is hardly possibly to provide matching references, especially due to the very general assumptions we use.

\subsection{Setup and notation} \label{sec: setup MCs}
Recall that $Y$ is a Markov chain on the topological space $\XX$. Denote by $\PP_\mu^Y:=\P_\mu(Y \in \cdot)$ the ``distribution'' of $Y$  started under a Borel measure $\mu$ on the space of sequences $\mathcal{B}(\XX)^{\otimes \N_0}$, and denote by $\EE_\mu^Y$ the ``expectation'' (integral) over $\PP_\mu^Y$. 

For the rest of Section~\ref{sec: setup MCs} we assume that $\mu$ a $\sigma$-finite non-zero invariant measure of $Y$. For any $x \in \XX^{\N_0}$, denote by $x_i$ the $i$-th coordinate of~$x$. Similarly, for any non-empty set $I \subset \N_0$, denote by $x_I$ the projection of $x$ onto $\XX^I$. Let $\theta$ be the {\it one-sided shift} operator on $\XX^{\N_0}$ defined by $(\theta x)_i :=x_{i+1}$ for $i \in \N_0$. Then~$\theta$ is a measure preserving transformation of the $\sigma$-finite measure space $(\XX^{\N_0}, \mathcal{B}(\XX)^{\otimes \N_0}, \PP_\mu^Y)$. 

For any set $C \in \mathcal{B}(\XX)^{\otimes \N_0}$, the shift $\theta$ defines the {\it first hitting time} $\tau_C$ of $C$,   the {\it first hitting mapping} $\varphi_C:=\theta^{\tau_C}$ of $C$, $\tilde \varphi_C:=\theta^{\tau_C \cdot \I_{\XX \setminus C}}$, and the {\it induced shift} $\theta_C:={(\varphi_C)|}_C$ on $C$; see the Appendix for details. For any $k \ge 1$ and $B \in \mathcal{B}(\XX)^{\otimes k}$, define the {\it cylindrical set}
\[
C_B:=\{x \in \XX^{\N_0}: (x_0, \ldots, x_{k-1}) \in B\}.
\]
We will use the short notation $\tau_B':=\tau_{C_B}$, which matches in the case $k=1$ the traditional probabilistic notation for the hitting time of the set $B$. For arbitrary $k$, we can think of $\tau_B'$ as of the hitting time of $B$ by the multivariate Markov chain $(Y_n, \ldots, Y_{n+k-1})_{n \ge 0}$ on $\XX^k$. 

We can use the new notation to write the entrance chain $Y^{\to A}$ into a set $A \in \BB( \XX)$ and the exit chain $Y^{A^c \to}$ from $A^c$, defined in Section~\ref{sec: entr exit}, as 
\begin{multline} \label{eq: positions chains}
\shoveleft{(Y_n^{A^c \to}, Y_n^{\to A}) =\bigl( \theta_{C_{A^c \times A}}^{n -1} \circ \tilde \varphi_{C_{A^c \times A}}(Y) \bigr)_{\{0,1\}}, 
\quad \text{ on } \{Y_k \in A \text{ i.o.}, Y_k \in A^c \text{ i.o.}\}, \, n \in \N.}
\end{multline}
We will also consider the {\it induced} sequence $Y^A$ obtained by restricting the chain $Y$ to $A$: 
\begin{equation} \label{eq: induced chain}
Y_n^A:=\bigl( \theta_{C_A}^{n-1} \circ \varphi_{C_A} (Y) \bigr)_0 \qquad \text{ on } \{Y_k \in A \text{ i.o.}\},  \, n \in \N_0.
\end{equation} 
It is easy to show, using the strong Markov property of $Y$, that $Y^A$ is a Markov chain on $A \cap N_A$  (cf.~Lemma~\ref{lem: Markov}). We will say that a Borel measure $\nu$ on $\XX$ is {\it proper}  for $Y^A$ if $\nu(N_A^c)=0$; we will start induced chains only under their proper measures. 

The powerful idea of ergodic theory is that the induced shift $\theta_C$ is a measure preserving transformation of the induced space $(C, \BB(\XX)^{\otimes \N_0} \cap C, {(\PP_\mu^Y)|}_C)$, under certain recurrence-type assumptions on $Y$ and $C$; see Lemmas~\ref{lem: induced}, ~\ref{lem: induced'}, and~\ref{lem: induced2}. Below we introduce the definitions needed to apply these general results of ergodic theory  in the context of Markov chains. We also refer the reader to  Kaimanovich~\cite[Section~1]{Kaimanovich} for a brief account of relevant results on invariant Markov shifts, and to Foguel~\cite{Foguel} for a detailed one.

The Markov chain $Y$ is called {\it recurrent starting under} $\mu$ if for every Borel set $B \subset \XX$ such that $\mu(B)<\infty$,  we have $\P_x(\tau_B'(Y)<\infty)=1$ for $\mu$-a.e.\ $x \in B$. Clearly, we can drop the assumption $\mu(B)<\infty$, as readily seen from  $\sigma$-additivity of $\mu$. It follows easily from invariance of $\mu$ that this definition is equivalent to $\P_x(\{Y_n \in B \text{ i.o.}\})=1$ for $\mu$-a.e.\ $x \in B$; see~\eqref{eq: io}. Following Kaimanovich~\cite{Kaimanovich}, we say that $Y$ is {\it transient starting under}~$\mu$ if for every $B \in \BB(\XX)$ such that $\mu(B)<\infty$,  we have $\P_x(\{Y_n \in B \text{ i.o.}\})=0$ for $\mu$-a.e.\ $x \in B$. We stress that for transient $Y$, it may be that $\P_{\mu|_B}(\{Y_n \in B \text{ i.o.}\})>0$ when $\mu(B)=\infty$. There is the usual recurrence--transience dichotomy, see Condition~\ref{cond: dychotomy} in Lemma~\ref{lem: cond rec} below.

We say that $Y$ is {\it topologically recurrent} if $\P_x(\tau_G'(Y)<\infty)=1$ for every open set $G \subset \XX$ and {\it every} $x \in G$. Warning: this definition matches the ergodic-theoretic one, while the Markov chains literature often defines topological recurrence by taking every $x \in \XX$ instead of every $x \in G$ (with $G \neq \varnothing$). In certain cases these two definitions are equivalent; see~\eqref{eq: conventional topological recurrence}.

Furthermore, we say that $Y$ is {\it ergodic starting under}~$\mu$ or, synonymously, $\mu$ is an {\it ergodic} invariant measure of $Y$, if  the ($\PP_\mu^Y$-preserving) shift $\theta$ is ergodic. The chain $Y$ is called {\it irreducible starting under} $\mu$ if every invariant set of $Y$ is $\mu$-trivial, that is for any $A \in \mathcal{B}(X)$, the equality $\P_x(Y_1 \in A) = \I_A(x) \Mod{\mu}$ implies that either $\mu(A)=0$ or $\mu(A^c)=0$. Warning: this shall not be confused with {\it  $\mu$-irreducibility}. We say that $Y$ is {\it topologically irreducible} if $\P_x(\tau_G'(Y) < \infty )>0$ for every $x \in \XX$ and every non-empty open set $G \subset \XX$. 

Let us give necessary and sufficient conditions for recurrence and ergodicity of $Y$. 

\begin{lemma} \label{lem: cond rec}
Let $Y$ be a Markov chain that takes values in a topological space $\XX$ and has a $\sigma$-finite invariant Borel measure $\mu$. The following statements hold true.

\begin{enumerate}[1)]
\item \label{cond: conserv} $Y$ is recurrent starting under $\mu$ iff the shift $\theta$ on $(\XX^{\N_0}, \mathcal{B}(\XX)^{\otimes \N_0}, \PP_\mu^Y)$ is conservative.

\item \label{cond: recur} $Y$ is recurrent starting under $\mu$ iff there exists a sequence of sets $\{B_n\}_{n \ge 1} \subset \mathcal B(\XX)$ such that $\XX= \cup_{n \ge 1} B_n \Mod{\mu}$, and $\P_{\mu|_{B_n}}(\tau_{B_n}'(Y) = \infty) =0$ and $\mu(B_n)<\infty$ for every $n \ge 1$.

\item \label{cond: recur finite}  $Y$ is recurrent starting under $\mu$ if for some $k \ge 1$ there exists a set $B \in \mathcal B(\XX)^{\otimes k}$ such that $\P_\mu(\tau_B'(Y) = \infty) = 0$ and $\P_\mu((Y_1, \ldots, Y_k) \in B) < \infty$. In particular, this holds if $\P_x(\tau_G'(Y)<\infty)=1$ for an open set $G \subset \XX$ of finite measure $\mu$ and every $x \in \XX$.

\item \label{cond: ergodic} $Y$ is ergodic and recurrent iff it is irreducible and recurrent, all properties starting under~$\mu$.

\item \label{cond: dychotomy} If $Y$ is irreducible, then it is either recurrent or transient, all properties starting under $\mu$.


\end{enumerate}
\end{lemma}
The assumption in Condition~\ref{cond: recur finite} is not vacuous since in general, a Borel measure on $\XX$ may be infinite on every non-empty open set. For example, if $\XX$ is separable,  take a sum of $\delta$-measures at the points of a dense countable subset of $\XX$; see Lemma~\ref{lem: Radon} below for conditions to exclude such pathologies. Condition~\ref{cond: dychotomy} extends the standard recurrence--transience dichotomy for countable space chains.
 
\begin{proof}
\ref{cond: conserv}) For the direct implication, note that since $\mu$ is $\sigma$-finite, $\XX^{\N_0}$ can be exhausted by countably many cylindrical sets $C_{B_n}$ with bases $B_n \in \BB(\XX)$ of finite measure. Each  set has measure $\PP_\mu^Y(C_{B_n})=\mu(B_n) < \infty$ and is recurrent for $\theta$ by recurrence of $Y$. Then $\theta$ is conservative by the Conditions for conservativity from the Appendix. For the reverse implication, every measurable cylindrical set $C_B$ is recurrent for $\theta$ by conservativity of $\theta$, hence $Y$ is recurrent. 

\ref{cond: recur}) These are the necessary and sufficient conditions appeared in the proof of Condition~\ref{cond: conserv}. 

\ref{cond: recur finite}) This holds by Maharam's recurrence theorem from the Appendix restated for $Y$ using Condition~\ref{cond: conserv}. Here  $\tau_{C_B} $ is finite $\PP_\mu^Y$-a.e.\ and  $\PP_\mu^Y(C_B)= \P_\mu((Y_1, \ldots, Y_k) \in B) < \infty$.

\ref{cond: ergodic}) The direct implication holds since every $\theta$-invariant cylindrical set $C_B$ with one-dimensional base $B \in \BB(\XX)$ is $\PP_\mu^Y$-trivial. The reverse one is in~\cite[Proposition~1.7]{Kaimanovich}. 

\ref{cond: dychotomy}) This is stated in~\cite[Theorem~1.2]{Kaimanovich}. Since there is neither formal proof nor exact reference given there, let us comment that this claim follows from  Foguel~\cite[Chapter~II]{Foguel}. In more detail, we have $\XX = C \cup D$, where $C$ and $D$ are respectively {\it conservative} and {\it dissipative} parts, which are disjoint and measurable. It follows (see~\cite[p.~17]{Foguel}) from  irreducibility of $Y$ that $\P_x(Y_1 \in C)=\I_C(x) \Mod{\mu}$. Then either $C=X \Mod{\mu}$, in which case $Y$ is recurrent by~\cite[Eq.~(2.4)]{Foguel}, or $D=X \Mod{\mu}$, in which case $Y$ is transient by repeating the argument after~\cite[Eq.~(2.4)]{Foguel} (for any $B \in \BB(\XX)$ such that $\mu(B)<\infty$, take $f= \I_B$ and $u = \I_{B_M}$, where $B_M:=\big \{x \in B: \sum_{k=0}^\infty \frac{d}{d \mu }\P_{{\mu|}_B} (Y^k \in dx) \le M \big \}$, and let $M \to \infty$). 
\end{proof}

We will also consider non-recurrent Markov chains, in which case we shall work with {\it invertible} measure preserving transformations on spaces of sequences. The shift $\theta$ on one-sided sequences in $\XX^{\N_0}$ is not invertible, therefore we shall extend time to negative integers. This corresponds to the {\it natural extension} in ergodic theory. It can be constructed using a standard argument (see Doob~\cite[Chapter~X.1]{Doob}) based on Kolmogorov's consistency theorem. However, this theorem may cease to hold if the space $\XX$ is not Polish (see Bogachev~\cite[Example~7.7.3]{Bogachev}). Therefore,  here we proceed differently.

Consider the space $\XX^\Z$ of two-sided sequences equipped with the $\sigma$-algebra $\BB(\XX)^{\otimes \Z}$. The {\it two-sided shift} $\bar \theta$ on $\XX^\Z$ is given by $(\bar \theta x)_i :=x_{i+1}$, where $x \in \XX^\Z$, $i \in \Z$. It is invertible and $\bar \theta^{-1}$ is measurable, since the collection of measurable sets $B \subset \XX^\Z$ such that $\theta(B)$ is measurable is a $\sigma$-algebra which contains all cylindrical sets.

Define the time-reversal operator $R: \XX^\Z \to \XX^\Z$, given by $R(x)_i:=x_{-i}$, where $i \in \Z$, for $x \in \XX^\Z$. Denote the projection onto $\XX^{\N_0}$ by $x^+:=(x_0, x_1, \ldots)$, and define the mapping $V: \XX^{\N_0} \to \XX^{-\N}$ by $V(y)_{-i}:=y_i$, where $i \in \N$, for $y \in \XX^{\N_0}$. It is easy to see that $R$, $V$, and $(\cdot )^+$ are measurable mappings. Note also that for every $x_0 \in \XX$, the restriction $V|_{C_{\{x_0\}}}$ is bijective and its inverse is measurable by the same argument as we used above for $\bar \theta^{-1}$.

Assume now that there exists a transition kernel $\hat P$ dual to $P$ with respect to $\mu$. This is always true when $\XX$ is Polish space, in which case a dual kernel is unique$\Mod{\mu}$ (we will always denote it by $\hat P$). Indeed, if $\mu$ is a probability measure, then this claim is nothing but the disintegration theorem combined with existence of regular conditional distributions for probability measures on Polish spaces; see Kallenberg~\cite[Theorems~6.3,~6.4, and~A1.2]{Kallenberg} or Aaronson~\cite[Theorem~1.0.8]{Aaronson}. This extends to $\sigma$-finite measures by $\sigma$-additivity. Recall that for a Polish space $\XX$, we have $\mathcal{B}(\XX) \otimes \mathcal{B}(\XX) = \mathcal{B}(\XX \times \XX)$, $\mathcal{B}(\XX)^{\otimes \Z} = \mathcal{B}(\XX^{\Z})$, etc. 

Denote by $\hat \PP_{x_0}$ the distribution on $\BB(\XX)^{\otimes \N_0}$ of a Markov chain with the transition kernel $\hat P$ and starting at  $x_0 \in \XX$; it exists by the Ionescu Tulcea theorem (\cite[Theorem~6.17]{Kallenberg}). We say that the probability measure 
\[
\bar \PP_{x_0}^Y := (\hat \PP_{x_0} \circ V^{-1}) \otimes \PP_{x_0}^Y \qquad \text{ on } \BB(\XX)^{\otimes \Z},
\]
is {\it an extended law} of $Y$ starting at $x_0$, where $\BB(\XX)^{\otimes \Z}$ is understood as  $\BB(\XX)^{\otimes (-\N)} \otimes \BB(\XX)^{\otimes \N_0}$.

We can assume that there is a sequence $\hat Y= (\hat{Y}_n)_{n \ge 0}$ of random elements of $\XX$, defined on the same measurable space $(\Omega, \mathcal F)$ with $Y$, such that $\hat Y_0=Y_0$ and for {\it every} $x_0 \in \XX$, a)~$\hat Y$ is a Markov chain with the transition kernel $\hat P$ under probability $\P_{x_0}$; and b) $ Y$ and $\hat Y$ are {\it independent}, as random elements of $(\XX^{\N_0}, \BB(\XX)^{\otimes \N_0})$, under probability $\P_{x_0}$. 

Then $\hat Y$ is a Markov chain dual to $Y$ with respect to~$\mu$. Moreover, $\hat Y$ is recurrent starting under $\mu$ iff so is $Y$; this follows from Conditions~\ref{cond: conserv} and~\ref{cond: recur} of recurrence and Remark~\ref{rem: measurable inverse}.b.

Indeed, we can define $Y$ and $\hat Y$ on the canonical space $(\Omega, \mathcal F)=(\XX^\Z, \BB(\XX)^{\otimes \Z})$ with $\omega =x$ by taking $Y(\omega)=x^+$ and $\hat Y(\omega)=R(x)^+$ for $x \in \XX^\Z$ and $\P_{x_0}=\bar \PP_{x_0}^Y$ for $x_0 \in \XX$. Then for any $x_0 \in \XX$ and any sets $B_1, B_2 \in \BB(\XX)^{\otimes \N_0} \cap C_{\{x_0\}}$, we have
\[
\P_{x_0} ( \hat Y \in B_1 , Y \in B_2)  = \P_{x_0} ( V(\hat Y) \in V(B_1) , Y \in B_2) = \bar \PP_{x_0}^Y \big( V(R(x)^+) \in V(B_1), x^+ \in B_2 \big),
\]
where the first equality holds true and makes sense since the function $V|_{C_{\{x_0\}}}$ is bijective and its inverse is measurable (hence $V(B_1)$ is measurable). By the equality $x=(V(R(x)^+), x^+)$ and the definition of the extended law $\bar \PP_{x_0}^Y$, we get
\[
\P_{x_0} ( \hat Y \in B_1 , Y \in B_2)   = \bar \PP_{x_0}^Y (V(B_1) \times B_2) = (\hat \PP_{x_0} \circ V^{-1})(V(B_1)) \cdot \PP_{x_0}^Y (B_2).
\]
Since 
$B_1 = \{x_0\} \times B_1'$ for some $B_1' \in \BB(\XX)^{\otimes \N}  $, we have $V^{-1}(V(B_1))= \XX \times B_1'$, hence $\hat \PP_{x_0}\big (V^{-1}(V(B_1)) \big) = \hat \PP_{x_0}(B_1) $ because $\hat \PP_{x_0}$ is supported on $C_{\{x_0\}}$. Therefore, we get
\[
\P_{x_0} ( \hat Y \in B_1 , Y \in B_2)  =\hat \PP_{x_0}  (B_1) \cdot  \PP_{x_0}^Y (B_2),
\]
which implies independence of $Y$ and $\hat Y$ under~$\P_{x_0}$, as required. 

Furthermore, arguing as above, we can obtain
\[
\bar  \PP_{x_0}^Y  = \P_{x_0} \big((V(\hat Y), Y) \in \cdot \big), \qquad x_0 \in \XX. 
\]
Notice that the function $x_0 \mapsto \bar \PP_{x_0}^Y (B)$ is measurable for every $B \in \BB(\XX)^{\otimes \Z}$ since the collection of the sets $B$ with this property is a $\sigma$-algebra (this follows from the monotone convergence theorem) which contains measurable rectangles. Then we can define
\begin{equation} \label{eq: decomposition}
\bar \PP_\mu^Y(B) := \int_\XX \P_{x_0} \big((V(\hat Y), Y) \in B \big) \mu(d x_0), \qquad B \in  \BB(\XX)^{\otimes \Z},
\end{equation}
an {\it extended law} of $Y$ starting under $\mu$, which is simply the law of $(V(\hat Y), Y)$ under $\P_\mu$. When  $\XX$ is a Polish space, such extended law always exists and is unique since the same holds for~$\hat P$.

We now claim that the two-sided shift $\bar \theta$ is an invertible measure preserving transformation of the $\sigma$-finite measure space $(\XX^\Z, \mathcal{B}(\XX)^{\otimes \Z}, \bar \PP_{\mu}^Y)$. Then so is the measurable mapping~$\bar \theta^{-1}$.

It suffices to show that $\bar \theta$ preserves the measure $\bar \PP_{\mu}^Y$ on cylindrical sets whose bases are measurable rectangles. To this end, it suffices to check that for any  sets $B_0, B_1, \ldots \in \BB(\XX)$ and any integers $1 \le k \le n$, we have
\begin{equation} \label{eq: 2sided shift}
\P_\mu(Y_0 \in B_0, \ldots, Y_n \in B_n) = \P_\mu(\hat Y_k \in B_0, \ldots, \hat Y_1 \in B_{k-1}, Y_0 \in B_k, \ldots, Y_{n-k} \in B_n). 
\end{equation}

Denote $f(x_0):= \P_{x_0}(Y_1 \in B_{k+1}, \ldots, Y_{n-k} \in B_n)$ if $k<n$ and $f (x_0):= 1$ if $k=n$. Then 
\begin{align*}
\P_\mu(Y_0 \in B_0, \ldots, Y_n \in B_n) &=\int_\XX \mu(x_0) \int_{B_k}f(x_k)\P_{x_0}(Y_0 \in B_0, \ldots, Y_k  \in d x_k) \\
&= \E_\mu [\I(Y_0 \in B_0, \ldots, Y_k  \in B_k) f(Y_k)]\\
&= \E_\mu [\I(\hat Y_k \in B_0, \ldots, \hat Y_0  \in B_k) f(\hat Y_0)]\\
&= \int_{B_k} \P_{x_0}(\hat Y_k \in B_0, \ldots, \hat Y_1  \in B_{k-1}) f(x_0) \mu (d x_0),
\end{align*}
where in the first equality follows from  the Markov property of $Y$ (under $\P_{x_0}$) and the third one follows from time-reversal equality~\eqref{eq: time-reversal}. This yields required equality~\eqref{eq: 2sided shift} by independence of $Y$ and $\hat Y$ under $\P_{x_0}$ for every $x_0 \in \XX$.

\subsection{Invariant measures of general induced, entrance, and exit Markov chains} \label{sec: Inducing MCs}

The proof of our first result shows that the method of inducing allows us to {\it compute} invariant measures of the Markov chains mentioned. 

Recall that the exit set $A^c_{ex}$ of the chain $Y$ from $A^c$ was defined in \eqref{eq: Aex},  and the entrance set $A_{en}$ of $Y$ into $A$ was defined, in the case when $Y$ has a dual, in \eqref{eq: Aen}.

\begin{theorem} \label{thm: inducing MC}
Let $Y$ be a Markov chain that takes values in a topological space $\XX$ and has a $\sigma$-finite invariant Borel measure $\mu$. Let $A \subset \XX$ be a Borel set.


\begin{enumerate}[1)]
\item Assume that $\mu(A)>0$. Then the induced chain $Y^A$ has a proper non-zero invariant Borel measure $\mu_A:=\mu|_A$ on $A$ if either of the following conditions is true:
\begin{enumerate}[a)]
\item $Y$ is recurrent starting under $\mu$;
\item $\P_{\mu_A} (\tau_A'(Y) = \infty)=0 $ and the same holds (instead of~$Y$) for a Markov chain $\hat Y$ that is dual to $Y$ with respect to $\mu$ and satisfies $\hat Y_0=Y_0$;
\item $\P_{\mu_A} (\tau_A'(Y) = \infty)=0 $ and $\mu(A)<\infty$;
\end{enumerate}
Moreover, if $Y$ is recurrent (resp.\ recurrent and ergodic) starting under $\mu$, then $Y^A$ is recurrent (resp.\ recurrent and ergodic) starting under $\mu_A$.
\item Assume that $\P_{\mu|_{A^c}}( Y_1 \in A)>0$. Then the entrance chain $Y^{\to A}$ and the exit chain $Y^{A^c \to}$ have respective proper non-zero invariant Borel measures 
$$
\mu_A^{entr}:=\int_{A^c} \P_x(Y_1 \in \cdot ) \mu(dx) \text{ on }A \quad \text{and} \quad \mu_{A^c}^{exit}(dx):= \P_x(Y_1 \in A) \mu(dx) \text{ on } A^c,
$$
if either of the following conditions is true:
\begin{enumerate}[a)]
\item $Y$ is recurrent starting under $\mu$;
\item $\P_{\mu|_{A^c_{ex}}}(\tau_{A_{en}}'(Y)=\infty)=\P_{\mu|_{A_{en}}}(\tau_{A^c_{ex}}'(Y)=\infty) = 0$ and the same holds (instead of~$Y$) for a Markov chain $\hat Y$ that is dual to $Y$ with respect to $\mu$ and satisfies $\hat Y_0=Y_0$;
\item $\P_{\mu|_{A^c_{ex}}}(\tau_A'(Y)=\infty)=\P_{\mu|_{A^c_{ex}}}(\tau_{A^c_{ex}}'(Y)=\infty) = 0$ and $\P_{\mu|_{A^c}}( Y_1 \in A) < \infty$.
\end{enumerate}
\end{enumerate}
Moreover, if $Y$ is recurrent (resp.\ recurrent and ergodic) starting under $\mu$, then $Y^{\to A}$ and $Y^{A^c \to}$ are recurrent (resp.\ recurrent and ergodic) starting respectively under $\mu_A^{entr}$ and $\mu_{A^c}^{exit}$.
\end{theorem}

Theorem~\ref{thm: inducing MC} applies to recurrent chains merely if $\P_{\mu|_{A^c}}( Y_1 \in A)>0$, which means that $Y$  can get from $A^c$ to $A$. However, we stress that $Y$ {\it is not required to be recurrent}; instead, there shall exist a dual chain $\hat Y$ satisfying Assumptions 1b and 2b.

The result of Part 1a is not new and it is especially well-known for finite $\mu$. However, the best explicit reference we found, Revuz~\cite[Proposition~2.9 in Chapter 3]{Revuz}, assumes that~$Y$ is Harris-recurrent (this is stronger than recurrence starting under $\mu$). We do not know references for transient $Y$. For Part~2b, we essentially give a second proof of Theorem~\ref{thm: duality}  using the inducing approach (Remark~\ref{rem: duality}  in Section~\ref{sec: duality} explains the assumptions), which we will need later for Part 2b of Proposition~\ref{prop: Kac MC}. 

\begin{remark} \label{rem: Polish}
The assumption of recurrence (in Parts a) is the strongest and the assumptions of Parts~b are the weakest, in the following precise sense.

a) Parts 1c and 2c are secondary but we stated them for the purpose of completeness (on practise, it can be difficult to check irreducibility of $Y$). Indeed, if $Y$ is recurrent starting under $\mu$, then Assumptions 1c and 2c are clearly satisfied. In the opposite direction, if $Y$ is irreducible starting under $\mu$, then either of these assumptions implies, by Condition~\ref{cond: dychotomy} in Section~\ref{sec: setup MCs}, that $Y$ is recurrent starting under $\mu$, since $Y$ cannot be transient. 

b) Assume that $\XX$ is a Polish space. Then the measure $\mu_A^{entr}$, as defined in Theorem~\ref{thm: inducing MC} in a more general setting, has the form given by~\eqref{eq: measures def}; either of Assumptions 1a or 1c implies Assumption 1b; and either of Assumptions 2a or 2c implies Assumption 2b.

Indeed, on a Polish space, there always exists a Markov chain $\hat Y$ dual to $Y$ (see Section~\ref{sec: setup MCs}). Then the formula for $\mu_A^{entr}$ simplifies immediately by the duality. Furthermore, if $Y$ is recurrent starting under $\mu$, then Assumption 1b for $Y$ is satisfied, and Assumption 2b for $Y$ holds by Condition~\ref{cond: conserv} in Section~\ref{sec: setup MCs}. Also, if $Y$ is recurrent starting under $\mu$, then so is the dual chain $\hat Y$ (see Section~\ref{sec: setup MCs}). Then Assumption 1b for $\hat Y$ holds true as just shown. By the same reasoning applied for the set $A^c$ instead of $A$, we see that Assumption 2b for $Y$ is satisfied since  $\P_{\mu|_A}( \hat Y_1 \in A^c)>0$ by the duality and we can write the sets $A_{en} = A_{en, Y}$ and $A^c_{ex}=(A^c)_{ex, Y}$ in terms of $\hat Y$ using equalities \eqref{eq: AenAex symm}. Finally, Assumption 1c (resp.\ 2c) implies Assumption 1b (resp.\ 2b) by Remark~\ref{rem: measurable inverse}.b.

c)  The results of Theorem~\ref{thm: inducing MC} (as well as of Theorem~\ref{thm: duality} and Proposition~\ref{prop: Kac MC}) are purely measure-theoretic and can be restated for any $\sigma$-algebra on $\XX$ instead of $\BB(\XX)$. The assumptions of Parts a, b, c correspond respectively to those of Lemmas~\ref{lem: induced2},~\ref{lem: induced},~\ref{lem: induced'}. 
\end{remark}

\begin{proof}
1) a) By Condition~\ref{cond: conserv} in Section~\ref{sec: setup MCs}, recurrence of $Y$ is equivalent to conservativity of the measure preserving shift $\theta$ on $(\XX^{\N_0}, \mathcal{B}(\XX)^{\otimes \N_0}, \PP_\mu^Y)$. We have $\PP_\mu^Y(C_A)=\mu(A)>0$ and by Lemma~\ref{lem: induced2}, the induced shift $\theta_{C_A}$ is a measure preserving conservative transformation of the induced space $(C_A, \mathcal{B}(\XX)^{\otimes \N_0} \cap C_A, (\PP_\mu^Y)|_{C_A})$. Then for any Borel set $B  \subset A$,
\begin{align*}
\mu_A(B) &= \P_\mu(Y_1 \in B) = (\PP_\mu^Y)|_{C_A}(C_B) \\
&= (\PP_\mu^Y)|_{C_A} \bigl(x \in \XX^{\N_0}: \theta_{C_A}(x) \in C_B \bigr) 
= \P_{\mu_A} \bigl({(\theta_{C_A}(Y))}_0 \in B \bigr) =\P_{\mu_A}(Y_1^A \in B),
\end{align*}
where in the last equality we used definition  \eqref{eq: induced chain} of $Y^A$ and the equality $\theta_{C_A}=(\varphi_{C_A})\bigl . \bigr |_{C_A}$. Thus, the measure $\mu_A$ is invariant for the induced chain $Y^A$. This measure is proper for $Y^A$ by implication \eqref{eq: io} and the fact that $(\PP_\mu^Y)|_{C_A}(\tau_{C_A}=\infty)=0$, which holds by conservativity of $\theta$. Clearly, $\mu_A$ is non-zero since $\mu(A)>0$.

Recurrence of $Y^A$ starting under $\mu_A$ follows trivially from that of $Y$ starting under $\mu$. It remains to obtain ergodicity of the induced chain from ergodicity and recurrence of $Y$. Use representation~\eqref{eq: induced chain} to write the law of the induced chain $Y^A$ starting under $\mu_A$ as $\PP_{\mu_A}^{Y^A} =(\PP_\mu^Y)|_{C_A} \circ \psi^{-1}$, where $\psi: C_A \to A^{\N_0}$ is defined by $\psi(x):= (x_0, (\theta_{C_A}(x))_0, (\theta_{C_A}^2(x))_0, \ldots)$. 

Note that $\psi(\theta_{C_A}(x))= \theta(\psi(x))$ for every $x \in C_A$, implying
\begin{equation} \label{eq: shifting induced}
\theta_{C_A}^{-1} (\psi^{-1}B) = \psi^{-1}(\theta^{-1}B), \qquad B  \in \BB( A)^{\N_0}.
\end{equation}
In particular, this yields that $\theta$ (restricted to $A^{\N_0}$) is measure preserving on $(A^{\N_0},\mathcal{B}(A)^{\otimes \N_0} , \PP_{\mu_A}^{Y^A})$. We already know this since $\mu_A$ is an invariant measure for the chain $Y^A$ and we have $\BB(A)^{\otimes \N_0} = \mathcal{B}(\XX)^{\otimes \N_0} \cap  A^{\N_0}$, which follows from the definition of the induced topology on $A$. To show ergodicity of $\theta$  on $A^{\N_0}$, consider an invariant set $B \in  \BB( A)^{\N_0}$, that is $\theta^{-1} B = B \Mod{\PP_{\mu_A}^{Y^A}}$ or, equivalently, $\psi^{-1} (\theta^{-1} B) = \psi^{-1} B \Mod{(\PP_{\mu}^Y)|_{C_A}}$. By \eqref{eq: shifting induced}, this gives $\theta_{C_A}^{-1} (\psi^{-1} B) = \psi^{-1} B \Mod{(\PP_\mu^Y)|_{C_A}}$, meaning that $\psi^{-1} B$ is an invariant set for $\theta_{C_A}$ on $C_A$. Since $\theta_{C_A}$ is ergodic by Lemma~\ref{lem: induced2}, the set $\psi^{-1}B$ is $(\PP_\mu^Y)|_{C_A}$-trivial, implying that $B$ is $\PP_{\mu_A}^{Y^A}$-trivial. This proves ergodicity of the induced Markov chain $Y^A$ starting under $\mu_A$. 

b) The two-sided shift $\bar \theta$ is an invertible measure preserving transformation of the  measure space $(\XX^\Z, \mathcal{B}(\XX)^{\otimes \Z}, \bar \PP_\mu^Y)$, and so is $\bar \theta ^{-1}$. Denote $\bar C_A:= \overline{C_A}$. This is a cylindrical set in $\XX^\Z$ with no constraints on negative coordinates. Hence
$\bar \PP_{\mu}^Y(\bar C_A) = \PP_\mu^Y(C_A)=\mu(A)>0$ and
\begin{align} \label{eq: zero}
\bar \PP_\mu^Y \bigl(\bar C_A  \setminus \cup_{k \ge 1} \bar \theta^{-k} (\bar C_A) \bigr)& = \bar \PP_{\mu}^Y(x\in \XX^\Z: x_0 \in A,\tau_{\bar C_A} = \infty) \notag \\ 
&= \PP_{\mu}^Y \bigl(x \in \XX^{\N_0}: x_0 \in A, \tau_{C_A} =\infty \bigr) = \P_{\mu_A}(\tau_A'(Y)=\infty)= 0.
\end{align}
In particular, by \eqref{eq: io} this implies that the measure $\mu_A$ is proper for $Y^A$. Similarly, use representation \eqref{eq: decomposition} to get
\[
\bar \PP_\mu^Y \bigl(\bar C_A \setminus \cup_{k \ge 1} \bar \theta^{k} (\bar C_A)  \bigr) = \bar \PP_\mu^Y \bigl(x: x_0 \in A, x_{-1}, x_{-2}, \ldots \not \in A \bigr) = \PP_\mu^{\hat Y} \bigl(x: x_0 \in A, \tau_{C_A} = \infty \bigr)=0.
\]

From Lemma~\ref{lem: induced}, the induced shift $\bar \theta_{\bar C_A}$ is measure preserving on $(\bar C_A, \mathcal{B}(\XX)^{\otimes \Z} \cap \bar C_A, (\bar \PP_\mu^Y)|_{\bar C_A})$. Hence for any Borel set $B \subset A$,
\[
\mu_A(B) = (\bar \PP_\mu^Y)|_{\bar C_A}(\bar C_B) = (\bar \PP_\mu^Y)|_{\bar C_A} \bigl(x: \bar \theta_{\bar C_A}(x) \in \bar C_B \bigr) = 
(\PP_\mu^Y)|_{C_A} \bigl(x: \theta_{C_A}(x) \in C_B \bigr).
\]
The last probability was already shown in (the proof of) Part 1a to be $\P_{\mu_A}(Y_1^A \in B)$. 

c) We have $\PP_\mu^Y(C_A)>0$ and $\PP_{\mu}^Y \bigl(x \in C_A, \tau_{C_A} =\infty \bigr) =0$ arguing as in Part 1b, and the claim follows as in Part 1a if we apply Lemma~\ref{lem: induced'} instead of Lemma~\ref{lem: induced2}.

2) a) We argue as above in Part 1a. The cylindrical set $C_{A^c \times A}$ with two-dimensional base $A^c \times A$ satisfies $\PP_\mu^Y(C_{A^c \times A}) = \P_{\mu|_{A^c}}( Y_1 \in A)>0$. The induced shift $\theta_{C_{A^c \times A}}$ on $(C_{A^c \times A}, \mathcal{B}(\XX)^{\otimes \N_0} \cap C_{A^c \times A}, (\PP_\mu^Y)|_{C_{A^c \times A}})$ is measure preserving and conservative. In particular, implication \eqref{eq: io} and the equality $(\PP_\mu^Y)|_{C_{A^c \times A}}(\theta_{C_{A^c \times A}}=\infty)=0$, which holds by conservativity of $\theta$, give us
\[
0 =(\PP_\mu^Y)|_{C_{A^c \times A}}(\{\theta^k \in C_{A^c \times A} \text{ i.o.}\}^c)=\P_\mu(\{Y_0 \in A^c, Y_1 \in A \} \cap \{Y_k \in A \text{ i.o.}, Y_k \in A^c \text{ i.o.}\}^c).
\]
Hence, by the definitions of the sets $N_A$, $N_{A^c}$ and the measures $\mu_{A^c}^{exit}$, $\mu_A^{entr}$,  we get
\begin{align*}
0 &=  \P_\mu(Y_0 \in A^c \setminus (N_A \cap N_{A^c}), Y_1 \in A ) + \P_\mu(Y_0 \in A^c, Y_1 \in A \setminus (N_A \cap N_{A^c})) \\
&= \mu_{A^c}^{exit}(A^c \setminus (N_A \cap N_{A^c})) + \mu_A^{entr}(A \setminus (N_A \cap N_{A^c})).
\end{align*}
Thus, both measures $\mu_{A^c}^{exit}$ and $\mu_A^{entr}$ are proper for the chains $Y^{\to A^c}$ and $Y^{A \to}$. They are non-zero by
\[
\mu_A^{entr}(A) = \mu_{A^c}^{exit}(A^c) = \P_\mu  (Y_0 \in A^c, Y_1 \in A) = \P_{\mu|_{A^c}}( Y_1 \in A)>0.
\]

We now give the key computation which shows the use of inducing in {\it finding} invariant measures. By Lemma~\ref{lem: induced2}, for any measurable set $B \subset A^c \times A$,
\begin{align*}
\P_\mu((Y_0, Y_1) \in B) = (\PP_\mu^Y)|_{C_{A^c \times A}}(C_B) &= (\PP_\mu^Y)|_{C_{A^c \times A}} \bigl(x \in \XX^{\N_0}: \theta_{C_{A^c \times A}}(x) \in C_B \bigr) \\
&= \P_{\mu} \bigl((Y_0, Y_1) \in A^c \times A, {(\theta_{C_{A^c \times A}}(Y))}_{\{0,1\}} \in B \bigr).
\end{align*}
Combining this with representation \eqref{eq: positions chains} and the identity $\theta_{C_{A^c \times A}}=(\varphi_{C_{A^c \times A}})\bigl . \bigr |_{C_{A^c \times A}}$, we obtain 
\[
\P_\mu((Y_0, Y_1) \in B) = \P_{\mu} \bigl((Y_0, Y_1) \in A^c \times A, (Y_2^{A^c \to}, Y_2^{\to A}) \in B \bigr).
\]

Let us take $B=A^c \times B_1$, where $B_1 \subset A$ is an arbitrary Borel set. Then 
\begin{align} \label{eq: entr chain starts}
\mu_A^{entr}(B_1) &= \P_{\mu} \bigl((Y_0, Y_1) \in A^c \times A, Y_2^{\to A} \in B_1 \bigr) \notag\\
&= \int_{A^c} \mu(d x_0) \int_A \P_{x_0}(Y_2^{\to A} \in B_1 | Y_1 = x_1) \P_{x_0} (Y_1 \in d x_1) \notag\\
&= \int_{A^c} \mu(d x_0) \int_A \P_{x_1}(Y_1^{\to A} \in B_1) \P_{x_0} (Y_1 \in d x_1) \notag\\
&=\int_A P_A^{entr}(x_1, B_1)  \mu_A^{entr}(dx_1), 
\end{align}
where in the third equality we used the Markov property of $Y$ and in the last one we used that the measure $\mu_A^{entr}$ is proper for the entrance chain $Y^{\to A}$ (its transition kernel is $P_A^{entr}$). Thus, $\mu_A^{entr}$ is invariant for $Y^{\to A}$. 

Similarly, let us take $B=B_0 \times A$, where $B_0 \subset A^c$ is an arbitrary Borel set. Then
\begin{align}  \label{eq: exit chain starts}
\mu_{A^c}^{exit}(B_0) &=\P_{\mu} \bigl((Y_0, Y_1) \in A^c \times A, Y_2^{A^c \to} \in B_0 \bigr) \notag\\
&= \int_{A^c} \mu(d x_0) \int_A \P_{x_1}(Y_1^{A^c \to} \in B_0) \P_{x_0} (Y_1 \in d x_1) \notag\\
&= \int_{A^c_{ex}} \P_{x_0} (Y_1 \in A) \mu(d x_0) \int_A \P_{x_1}( Y_1^{A^c \to} \in B_0 ) \P_{x_0} (Y_1 \in d x_1| Y_1 \in A) \notag\\
&= \int_{A^c_{ex}} P_{A^c}^{exit}(x_0, B_0)  \mu_{A^c}^{exit}(dx_0), 
\end{align}
where in the last equality we used formula \eqref{eq: exit chain} for the transition kernel $P_{A^c}^{exit}$ of the entrance chain $Y^{A^c \to}$ and the fact that the measure $\mu_{A^c}^{exit}$ is proper for  $Y^{A^c \to}$. Thus, $\mu_{A^c}^{exit}$ is invariant for $Y^{A^c \to}$.

Recurrence of the entrance chain trivially follows from conservativity of the induced transformation $\theta_{C_{A^c \times A}}$ if we argue as above in \eqref{eq: entr chain starts} to start $Y^{\to A}$ under  $\mu_A^{entr}$. Similarly, the exit chain is recurrent starting under $\mu_{A^c}^{exit}$; cf.~\eqref{eq: exit chain starts}.

As for ergodicity, define the functions $\psi_0: C_{A^c \times A} \to (A^c)^\N$ and $\psi_1: C_{A^c \times A} \to A^\N$ by
\[
\psi_i(x):= (x_i, (\theta_{C_{A^c \times A}}(x))_i, (\theta_{C_{A^c \times A}}^2(x))_i, \ldots), \quad x \in C_{A^c \times A}, \, i \in \{0,1\}.
\]
The entrance chain $(Y^{\to A}_n)_{n \ge 1}$ starts from $Y_1^{\to A}$, which is $Y_1$ on the event  $\{Y_0 \in A^c, Y_1 \in A\}$. Since for any Borel set $B \subset A$ it is true that
\[
\mu_A^{entr}(B) = \P_\mu(Y_0 \in A^c, Y_1 \in B) = \PP_\mu^Y\bigl(x \in C_{A^c \times A}: x_1 \in B \bigr),
\]
we can write the law of  $(Y^{\to A}_n)_{n \ge 1}$ with $Y_1^{\to A}$ distributed according to $\mu_A^{entr}$ as $(\PP_\mu^Y)|_{C_{A^c \times A}} \circ \psi_1^{-1}$. Similarly, the law of the exit chain $(Y^{A^c \to}_n)_{n \ge 1}$ with $Y_1^{A^c \to}$ following $\mu_A^{exit}$ is $(\PP_\mu^Y)|_{C_{A^c \times A}} \circ \psi_0^{-1}$. Then ergodicity of the chains $Y^{\to A}$ and $Y^{A^c \to}$ follows from that of $Y$ exactly as in Part~1.

b) We prove as in Part 1b. The cylindrical set $\bar C_{A^c \times A}$ satisfies $\bar \PP_\mu^Y(\bar C_{A^c \times A})= \P_{\mu|_{A^c}}(Y_1 \in A)>0$. Arguing similarly to \eqref{eq: zero} and using that $\tau_{A^c \times A}'(Y) = \tau_{A^c_{ex} \times A_{en}}'(Y) \Mod{\P_\mu}$, we get
\begin{align} \label{eq: bar C hit} 
&\mathrel{\phantom{=}}  \bar \PP_\mu^Y \bigl(\bar C_{A^c \times A}  \setminus \cup_{k \ge 1} \bar \theta^{-k} (\bar C_{A^c \times A}) \bigr) \notag \\
&= \P_{\mu} \bigl(Y_0 \in A^c, Y_1 \in A, \tau_{A^c \times A}'(Y ) =\infty \bigr) \notag \\
&= \P_\mu(Y_0 \in A^c_{ex}, Y_1 \in A_{en}, Y_2, Y_3, \ldots \not \in A^c_{ex}) \notag \\
&\mathrel{\phantom{=}} + \sum_{k=1}^\infty \P_\mu\bigl(Y_0 \in A^c_{ex}, Y_1 \in A_{en}, Y_2, \ldots, Y_k \not \in A^c_{ex}, Y_{k+1} \in A^c_{ex}, Y_{k+2}, Y_{k+3}, \ldots \not \in A_{en}\bigr) \notag \\ 
&\le \P_{\mu|_{A_{en}}}(\tau_{A^c_{ex}}'(Y)=\infty) + \infty \cdot \P_{\mu|_{A^c_{ex}}} (\tau_{A_{en}}'(Y) = \infty) =0.
\end{align}
In particular, this implies (as in Part 1b) that both measures $\mu_{A^c}^{exit}$ and $\mu_A^{entr}$ are proper for the chains $Y^{\to A^c}$ and $Y^{A \to}$.

Further, by representation \eqref{eq: decomposition} and invariance of  the right two-sided shift $\bar \theta^{-1}$, we get
\begin{align*}
&\mathrel{\phantom{=}}  \bar \PP_\mu^Y \bigl(\bar C_{A^c \times A} \setminus \cup_{k \ge 1} \bar \theta^{k} (\bar C_{A^c \times A})  \bigr) \\
&= \bar \PP_\mu^Y \bigl(x \in \XX^\Z: (x_0,x_1) \in A^c \times A, (x_{-1}, x_0)\not \in A^c \times A, (x_{-2}, x_{-1})\not \in A^c \times A, \ldots  \bigr) \\
&= \bar \PP_\mu^Y \bigl(x \in \XX^\Z: (x_0,x_{-1}) \in A \times A^c, (x_{-1}, x_{-2})\not \in A \times A^c, (x_{-2}, x_{-3})\not \in A \times A^c, \ldots  \bigr) \\
&= \P_\mu \bigl(\hat Y_0 \in A, \hat Y_1 \in A^c, \tau_{A \times A^c}'(\hat Y) = \infty \bigr) =0,
\end{align*}
where the last equality follows from $\tau_{A \times A^c}'(\hat Y) = \tau_{ A_{en} \times A^c_{ex} }'(\hat Y) \Mod{\P_\mu}$ exactly as in \eqref{eq: bar C hit}.

By Lemma~\ref{lem: induced}, the induced two-sided shift $\bar \theta_{\bar  C_{A^c \times A}  }$ is a measure preserving transformation of $(\bar  C_{A^c \times A}, \mathcal{B}(\XX)^{\otimes \Z} \cap \bar  C_{A^c \times A}, (\bar \PP_\mu^Y)|_{\bar C_{A^c \times A}})$. Then the equality $(\bar \PP_\mu^Y)|_{\bar C_{A^c \times A}}(\bar C_B) = (\bar \PP_\mu^Y)|_{\bar C_{A^c \times A}} \bigl(\bar \theta_{\bar C_{A^c \times A}} \in \bar C_B \bigr)$ holds for any measurable set $B \subset A^c \times A$, and thus
\[
\P_\mu((Y_0, Y_1) \in B) = (\PP_\mu^Y)|_{C_{A^c \times A}} \bigl(\theta_{C_{A^c \times A}} \in C_B \bigr).
\]
This equality implies invariance of $\mu_{A^c}^{exit}$ and $\mu_A^{entr}$ as required, as already shown in Part~2a. 

c) We have $\PP_\mu^Y(C_{A^c \times A})>0$ and $\PP_{\mu}^Y \bigl(x \in C_{A^c \times A}, \tau_{C_{A^c \times A}} =\infty \bigr) =0$ arguing as in Part 2b, and the claim follows as in Part 2a if we apply Lemma~\ref{lem: induced'} instead of Lemma~\ref{lem: induced2}.
\end{proof}

Notice that we could have proved Part 2 by applying the result of Part 1  to the bivariate chain $Z=((Y_n, Y_{n+1}))_{n \ge 0}$, which
takes values in the space $\XX \times \XX$ (equipped with $\BB(\XX) \otimes \BB(\XX)$, cf.~Remark~\ref{rem: Polish}.c), has an invariant measure $\P_\mu((Y_0, Y_1) \in \cdot)$, and satisfies the relation
\begin{multline} \label{eq: reduction 2d}
\shoveleft{Z^{A^c \times A}_{n-1}=(Y^{A^c \to}_n, Y^{\to A}_n) \qquad \quad \text{ on } \{Y_0 \in A^c, Y_1 \in A\}, \, n \in \N,}
\end{multline}
following from \eqref{eq: positions chains}. However, the proof of Part 2 we presented appears to be the most natural since the full trajectory of the Markov chain $Y$ is already at our disposal due to the setup used. Moreover, our general proof allows other applications of inducing (to be considered in~\cite{Vysotsky2018}), based on the use of stopping times which are more complicated than the hitting times $T^{\to A}$ considered here,  where reduction to a multivariate chain is not possible.

\medskip

We next present occupation time formulas for lifting invariant measures of the induced and entrance chains to recover the invariant measure of the underlying Markov chain. The assumptions of the following result are stronger than those of respective parts of Theorem~\ref{thm: inducing MC}.

\begin{proposition} \label{prop: Kac MC}
Let $Y$ be a Markov chain that takes values in a topological space $\XX$ and has a $\sigma$-finite non-zero invariant measure $\mu$. Let $A \in \BB( \XX)$ be such that $\P_\mu(\tau_A'(Y)=\infty)=0$. 
\begin{enumerate}[1)]
\item We have
\begin{equation} \label{eq: Kac MC 1}
\P_\mu(Y \in E) = \E_{\mu_A} \! \left [ \sum_{k=0}^{\tau_A'(Y) - 1} \I\bigl((Y_k, Y_{k+1}, \ldots) \in E \bigr) \right ] \!, \qquad E \in \mathcal{B}(\XX)^{\otimes \N_0},
\end{equation}
if either of the following conditions is true:
\begin{enumerate}[a)]
\item $Y$ is recurrent starting under $\mu$; 
\item there is a dual chain $\hat Y$ satisfying  $\P_{\mu_A} (\tau_A'(\hat Y) = \infty)=0 $.
\end{enumerate}
\item We have
\begin{equation} \label{eq: Kac MC 2}
\P_\mu(Y \in E) = \E_{\mu_A^{entr}} \! \left [ \sum_{k=0}^{T_1^{\to A} - 1} \I\bigl((Y_k, Y_{k+1}, \ldots) \in E \bigr) \right ] \!, \qquad E \in \mathcal{B}(\XX)^{\otimes \N_0},
\end{equation}
if $\P_{\mu}(\tau_{A^c}'(Y)=\infty)=0$ and either of the following conditions is true:
\begin{enumerate}[a)]
\item $Y$ is recurrent starting under $\mu$; 
\item there is a dual chain $\hat Y$ satisfying $\P_{\mu|_{A^c_{ex}}}(\tau_{A_{en}}'( \hat Y)=\infty)=\P_{\mu|_{A_{en}}}(\tau_{A^c_{ex}}'(\hat Y)=\infty) = 0$.
\end{enumerate}
\end{enumerate}
\end{proposition}

Thus, we can recover $\mu$ from $\mu_A$ or $\mu_A^{entr}$ using \eqref{eq: Kac MC 1} or \eqref{eq: Kac MC 2} with $E=C_B$ for $B \in \mathcal{B}(\XX)$.

\begin{proof}
1) The  equality $\P_\mu(\tau_A'(Y)=\infty)=0$ ensures that $\mu(A)>0$ by
\begin{equation} \label{eq: A non-zero}
0< \mu(\XX) = \sum_{k=1}^\infty \P_\mu(\tau_A'(Y)=k) \le \sum_{k=1}^\infty \P_\mu(Y_k \in A) = \infty \cdot \mu(A).
\end{equation}
Thus, Assumptions 1a and 1b of the proposition are respectively stronger than Assumptions~1a and 2b of Theorem~\ref{thm: inducing MC}.

b) As we have seen in the proof of Part 1b of Theorem~\ref{thm: inducing MC}, the two-sided shift $\bar \theta$ is an invertible measure preserving transformation of $(\XX^\Z, \mathcal{B}(\XX)^{\otimes \Z}, \bar \PP_\mu^Y)$ satisfying the assumptions of Lemma~\ref{lem: induced}. For the cylindrical set $\bar C_A$, we have $\bar \PP_\mu^Y(\bar C_A) = \mu(A)>0$ and $\bar \PP^Y_{\mu} (\tau_{\bar C_A} = \infty) = \P_\mu(\tau_A'(Y)=\infty) = 0$. Thus, all the assumptions of Lemma~\ref{lem: induced3} are satisfied. From \eqref{eq: Kac} and \eqref{eq: decomposition} it follows that for any $E \in \mathcal{B}(\XX)^{\otimes \N_0}$,
\[
\PP_\mu^Y(E)= \bar \PP_\mu^Y(\bar E) = \int_{\bar C_A} \! \left [ \sum_{k=0}^{\tau_{\bar C_A}(x)-1} \I(\bar \theta^k x \in \bar E) \right ] \! \bar \PP_\mu^Y(dx) = \int_{C_A} \! \left[ \sum_{k=0}^{\tau_{C_A}(x)-1} \I(\theta^k x \in E ) \right] \!\PP_\mu^Y(dx).
\]
This implies \eqref{eq: Kac MC 1} by the equality $\tau_{C_A}=\tau_A'$, as required.

 a) As we have seen in the proof of Part 1a of Theorem~\ref{thm: inducing MC}, the one-sided shift $\theta$ is a measure preserving conservative transformation of $(\XX^{\N_0}, \mathcal{B}(\XX)^{\otimes \N_0}, \PP_\mu^Y)$. For the cylindrical set $C_A$, we have $\PP_\mu^Y(C_A) = \mu(A)>0$ and $\PP^Y_{\mu} (\tau_{C_A} = \infty) = \P_\mu(\tau_A'(Y)=\infty) = 0$. Thus, all the assumptions of Lemmas~\ref{lem: induced2} and~\ref{lem: induced3} are satisfied for the one-sided shift $\theta$ and the set $C_A$, and \eqref{eq: Kac MC 1} follows directly from formula \eqref{eq: Kac} as shown in the proof of Part 1b above.

2) Note that $\P_\mu(Y_0 \in A^c, Y_1 \in A)>0$. In fact, relation \eqref{eq: A non-zero} applied for $A^c$ instead of $A$ and equality $\P_{\mu}(\tau_{A^c}'(Y)=\infty)=0$ imply that $\mu(A^c)>0$. Then the assertion follows from the equality $\P_{\mu_{A^c}}(\tau_A'(Y)=\infty)=0$ by an argument analogous to~\eqref{eq: A non-zero}. Thus, Assumptions 2a and 2b of the proposition are respectively stronger than Assumptions 2a and 2b of Theorem~\ref{thm: inducing MC}.

Then for the cylindrical set $C_{A^c \times A}$, we have $\PP_\mu^Y(C_{A^c \times A}) = \P_\mu(Y_0 \in A^c, Y_1 \in A)>0$ and $\PP^Y_{\mu} (\tau_{C_{A^c \times A}} = \infty) = 0$. The latter equality follows from the assumptions $\P_{\mu}(\tau_A'(Y)=\infty)=0$ and $\P_{\mu}(\tau_{A^c}'(Y)=\infty)=0$ by the same argument as in \eqref{eq: bar C hit}.

a) All the assumptions of Lemmas~\ref{lem: induced2} and~\ref{lem: induced3} are satisfied for the one-sided shift $\theta$ on $(\XX^{\N_0}, \mathcal{B}(\XX)^{\otimes \N_0}, \PP_\mu^Y)$ and the set $C_{A^c \times A}$. Then for any $ E \in \mathcal{B}(\XX)^{\otimes \N_0}$, we have
\[
\PP_\mu^Y(E) = \int \limits_{C_{A^c \times A}}  \!\! \left [ \sum_{k=1}^{\tau_{C_{A^c \times A}}(x)} \I(\theta^k x \in E ) \right ]  \!\PP_\mu^Y(dx)
= \int \limits_{A^c} \E_{x_0} \!\! \left [ \sum_{k=1}^{\tau_{A^c \times A}'(Y)} \I(\theta^k Y \in E, Y_1 \in A ) \right ] \! \mu(d x_0),
 \]
where in the first equality we applied \eqref{eq: Kac} and shifted the summation indices by one using  invariance of $(\PP_\mu^Y)|_{C_{A^c \times A}}$ under the induced shift $\theta_{C_{A^c \times A}}$. By the Markov property, we get
\[
\PP_\mu^Y(E) =\int_{A^c} \mu(d x_0) \int_A \E_{x_1} \!\! \left [\sum_{k=0}^{\tau_{A^c \times A}'(Y)} \I(\theta^k Y \in E) \right] \!\P_{x_0}(Y_1 \in d x_1).
\]
This implies \eqref{eq: Kac MC 2}  by definition of $\mu_A^{entr}$, the fact that this measure is proper for the entrance chain $T^{\to A}$, and  the equality $T_1^{\to A} = \tau_{A^c \times A}'(Y)  + 1$ on $\{Y_0 \in A,Y_k \in A \text{ i.o.}, Y_k \in A^c \text{ i.o.}\}$. 

b) The two-sided shift $\bar \theta$ on $(\XX^\Z, \mathcal{B}(\XX)^{\otimes \Z}, \bar \PP_\mu^Y)$ and the set $\bar C_{A^c \times A}$ satisfy the assumptions of Lemmas~\ref{lem: induced} and~\ref{lem: induced3}. Then equality \eqref{eq: Kac MC 2} follows from a computation similar to those in the proofs of Parts 1b and 2a above.
\end{proof}

\subsection{Existence and uniqueness for weak Feller chains} \label{sec: MC e and u}
In this section we apply the general results and ideas developed in Section~\ref{sec: Inducing MCs} for a more specific class of topologically recurrent weak Feller Markov chains. For existence and uniqueness results on invariant measures under much stronger assumptions on $Y$, such as strong Feller or Harris properties or $\psi$-irreducibility, see Foguel~\cite[Chapters~IV and VI]{Foguel} and Meyn and Tweedie~\cite{MeynTweedie}.

We assume throughout that $\XX$ is a metric space. A Markov chain $Y$ with values in $\XX$ is called  {\it weak Feller} if its transition kernel $P(x, \cdot)$ is weakly continuous in $x$. A Borel measure on $\XX$ is called {\it locally finite} if every point of $\XX$ admits an open neighbourhood of finite measure. Such measures are finite on compact sets. Also, they are $\sigma$-finite if $\XX$ is separable  since in this case the space can be represented as a countable union of open balls of finite measure.  Locally finite measures on Polish spaces are often called {\it Radon}. 

The main result of the section (supplemented by Proposition~\ref{prop: Skorokhod} below) is as follows.



\begin{theorem} \label{thm: inducing bijection}
Let $Y$ be a topologically irreducible topologically recurrent weak Feller Markov chain that takes values in a Polish space $\XX$.  Let $A \subset \XX$ be a Borel set with $\Int(A) \neq \varnothing$. 
\begin{enumerate}[1)]
\item The mapping $\mu \mapsto \mu_A$ is a bijection between the sets of locally finite Borel invariant measures of the chain $Y$ on $\XX$ and of the induced chain $Y^A$ on $A$. 
\item Assume that there exists an $x \in \Int(A^c)$ such that $\P_x(Y_1 \in \Int(A))>0$. Then the mappings $\mu \mapsto \mu_A^{entr}$ and $\mu \mapsto \mu_{A^c}^{exit}$ (defined in \eqref{eq: measures def}) are bijections between the sets of locally finite Borel invariant measures of the chain $Y$ on $\XX$ and, respectively, of the entrance chain $Y^{\to A}$ on $A$ and  the exit chain $Y^{A^c \to}$ on $A^c$.
\end{enumerate}
\end{theorem}

When this paper was nearly finished, we found a work by Skorokhod with a uniqueness result~\cite[Theorem~3]{Skorokhod} very similar to Part 1 of our Theorem~\ref{thm: inducing bijection} but presented in a more complicated way (for the purpose of proving existence, as explained below) and under the assumption that $\XX$ is {\it locally compact}. In this setting, the results are equivalent\footnote{\label{footn} Let us explain this, for the purpose of exposing a useful ``smoothing'' trick of Skorokhod. Take a continuous non-zero function $g$ on $\XX$ and put $A:=\Cl (\{x \in \XX: g(x)>0\})$. Let $(\alpha_n)_{n \ge 0}$ be a sequence of i.i.d.\ random variables distributed uniformly over $[0,1]$ and independent with $Y$ under $\P_x$ for every $x \in \XX$. Denote $
\tau_g:=\min\{n \ge 1: \alpha_n \le g(Y_n^A)\}$ and define the transition kernel $P_g(x, dy):=\P_x(Y_{\tau_g}^A \in dy)$ for $x, y \in A$, considered in~\cite[Theorem~3]{Skorokhod}. Define the Markov chain $Z:=(Z_n)_{n \ge 0}:=(\alpha_n, Y_n^A)_{n \ge 0}$ on $\XX':=[0,1] \times A$ and put $A':=\{(a,x) \in \XX': a \le g(x)\}$. 

It is easy to see that the chain $Z$ on the Polish space $\XX'$ satisfies the assumptions of Theorem~\ref{thm: inducing bijection}. Then $d a \otimes \mu_A(dx)$ is an invariant locally finite measure for $Z$ iff so is the measure ${(d a \otimes \mu_A(dx))}|_{A'}$ for the induced chain $Z^{A'}$. This yields, by integrating in $a$ over $[0,1]$, that the mapping $\mu_A \mapsto \mu_g$, given by $\mu_g(x):=g(x) \mu_A(dx)$ for $ x \in A$, is a bijection between the sets of locally finite measures invariant under the transition kernels $P_g$ and that  of $Y^A$. Using Part~1 of Theorem~\ref{thm: inducing bijection} once again establishes~\cite[Theorem~3]{Skorokhod}. 

The proof in the other direction is by approximation argument. 
}.

It is remarkable that under the assumptions of Theorem~\ref{thm: inducing bijection}, the chain $Y$ may have two non-proportional invariant measures even if the space $\XX$ is compact! This is shown in the nice example of Carlsson~\cite[Theorem~1]{Carlsson} whose assumptions are satisfied by Lemma~\eqref{lem: conventional topological recurrence}; also see Skorokhod~\cite[Example 1]{Skorokhod}. Furthermore, the condition $\P_x(Y_1 \in \Int(A))>0$ for an $x \in \Int(A^c)$ excludes the case when the chain can enter $A$ from $A^c$ only through $\partial A$. The weak Feller property is required only to get surjectivity of the mappings. 
Note that we cannot use the duality of Section~\ref{sec: duality} to infer the result of  Part 2 on the exit chain from that on the entrance chain since we cannot ensure that the dual chain $\hat Y$ is weak Feller.

The main use of Theorem~\ref{thm: inducing bijection} is when the initial chain $Y$ has a known unique (up to multiplication by positive constant) invariant measure. Our main application, given below in Section~\ref{Sec: Application to RWs}, concerns random walks on $\R^d$ with the invariant Haar measure. 

On the other hand, Theorem~\ref{thm: inducing bijection} can be used to prove existence of an invariant measure for the chain $Y$. We will give a result in this direction (Proposition~\ref{prop: existence}), which easily follows from the next statement. Recall that subsets of $\XX$ are equipped with the subspace topology. 

\begin{lemma} \label{lem: weak Feller}
Under the respective assumptions of Theorem~\ref{thm: inducing bijection}, the induced chain $Y^A$ and the entrance chain $Y^{\to A}$ are weak Feller on $A$ if $\P_x(Y_1 \in \partial A)=0$ for every $x \in \XX$.
\end{lemma}
\begin{cor*}
For any topologically recurrent random walk $S$ on $\R$ with continuous distribution of increments, the chain $O$ of overshoots at up-crossings of the zero level is weak Feller.
\end{cor*}

We postpone the proof of the lemma for a moment. The corollary follows immediately using that $S$ is weak Feller by $\P_x(S_1 \in \cdot) = \P(x + X_1  \in \cdot)$ for $x \in \XX$.

\begin{remark} \label{rem: not Feller}
In general, the chains $Y^A$ and $Y^{\to A}$ may not be weak Feller. For example, consider a recurrent random walk $S$ on $\R$ whose distribution of increments is continuous except for an atom at $-1$. Take $A=[0, \infty)$. Then $\P_1(S_1^A = 0) = \P_1(S_1 =0)>0$ but for any $0 \le x < 1$, the distribution of $S_1^A$ is continuous and thus $\P_x(S_1^A = 0) = 0$. Similar examples can be constructed for the chain $O=S^{\to A}$ of overshoots.
\end{remark}


We get the following existence result combining Theorem~\ref{thm: inducing bijection} with Lemma~\ref{lem: weak Feller} and the classical Bogolubov--Krylov theorem.

\begin{proposition} \label{prop: existence}
Let $Y$ be a topologically recurrent topologically irreducible weak Feller Markov chain that takes values in a locally compact Polish space $\XX$. Assume that there exists a compact set $A \subset \XX$ such that $\Int(A) \neq \varnothing$ and $\P_x(Y_1 \in \partial A)=0$ for every $x \in \XX$. Then the chain $Y$ has a locally finite non-trivial invariant measure.
\end{proposition}

This in turn yields, by Skorokhod's ``smoothing'' trick explained in Footnote~\ref{footn}, a stronger existence result (which is not explicit in~\cite{Skorokhod} but it was surely evident to Skorokhod, cf.~\cite[Lemma~4]{Skorokhod}). It was also obtained by Lin~\cite[Theorem~5.1]{Lin} using functional-analytic approach.

\begin{proposition}[\bf{Lin~\cite{Lin}; Skorokhod~\cite{Skorokhod}}] \label{prop: Skorokhod}
Every topologically recurrent topologically irreducible weak Feller Markov chain that takes values in a locally compact Polish space has a locally finite non-zero invariant measure.
\end{proposition}

Finally, we mention a more recent paper by Szarek~\cite{Szarek} on existence of invariant probability measures for weak Feller chains on non-locally compact spaces.

\medskip

Before proceeding to the proof of Theorem~\ref{thm: inducing bijection}, we give two simple auxiliary results.

\begin{lemma} \label{lem: Radon}
Let $Y$ be a topologically irreducible weak Feller Markov chain taking values in a metric space $\XX$. An invariant Borel measure $\mu$ of $Y$ is locally finite if and only if it is finite on some non-empty open set.
\end{lemma}

\begin{proof}
The necessary condition is trivial. To prove the sufficient one, assume that $G$ is a non-empty open subset of $\XX$ satisfying $ \mu(G)<\infty$. By topological irreducibility of $Y$, for any $x \in \XX$ there exists an $n=n(x) \ge 1$ such that $\P_x(Y_n \in G)>0$. It follows by a simple inductive argument that the $n$-step transition kernel $\P_x(Y_n \in \cdot)$ is weakly continuous in~$x$. Indeed, for any continuous bounded function $f: \XX \to \R$, we have 
\[
\E_x f(Y_n) = \int_\XX \E_y f(Y_{n-1}) \P_x( Y_1 \in dy), \qquad x \in \XX
\]
by the Chapman--Kolmogorov equation. The integrand is a continuous bounded function by assumption of induction, and so is the integral since $Y$ is weak Feller. 

Then there is an open neighbourhood $U_x$ of $x$ such that $\P_y(Y_n \in G) \ge \frac12 \P_x(Y_n \in G)$ for every $y \in U_x$. By invariance of $\mu$, this gives 
\begin{equation} \label{eq: finiteness}
\infty > \mu(G) = \int_\XX \P_y(Y_n \in G) \mu(dy) \ge \int_{U_x} \P_y(Y_n \in G) \mu(dy) \ge \frac12 \P_x(Y_n \in G) \mu(U_x),
\end{equation}
implying finiteness of $\mu(U_x)$, as required.
\end{proof}

\begin{lemma} \label{lem: conventional topological recurrence}
Let $Y$ be a topologically irreducible topologically recurrent weak Feller Markov chain that takes values in a metric space $\XX$. Then
\begin{equation} \label{eq: conventional topological recurrence}
\P_x(\tau_G'(Y) <\infty)=1 \text{ for every } x \in \XX \text{ and non-empty open }G \subset \XX.
\end{equation}
\end{lemma}
\begin{proof}
As in the proof of Lemma~\ref{lem: Radon}, by topological irreducibility and weak Feller property of $Y$ we can find an open neighbourhood $U$ of $x$ such that $\inf_{y \in U} \P_y(\tau_G'(Y) <\infty)>0$. The claim now follows by topological recurrence and the strong Markov property of the chain $Y$, which returns to $U$ $\P_x$-a.s.
\end{proof}

\begin{proof}[{\bf Proof of Theorem~\ref{thm: inducing bijection}}] 
First of all, the sets of invariant measures of the chains $Y^A$, $Y^{\to A}$, and $Y^{A^c \to}$ are non-empty since they contain the zero measure (an even non-zero ones, by Proposition~\ref{prop: Skorokhod}).

Let $\mu$ be a non-zero locally finite Borel invariant measure of the Markov chain $Y$. Then the measure $\mu_A$ is Borel and locally finite on $A$, and so are the measures $\mu_A^{entr}$ on $A$ and $\mu_{A^c}^{exit}$ on $A^c$ as follows from the inequalities $\mu_A^{entr} \le \mu_A$ and $\mu_{A^c}^{exit} \le \mu_{A^c}$. Further, $\mu$ is $\sigma$-finite as a locally finite measure on a Polish space. By choosing an open set $G$ in~\eqref{eq: conventional topological recurrence} of finite measure, we conclude that $Y$ is recurrent starting under $\mu$ by Condition~\ref{cond: recur finite} in Section~\ref{sec: setup MCs}. Then Theorem~\ref{thm: inducing MC} applies, and the measure $\mu_A$, $\mu_A^{entr}$, and  $\mu_{A^c}^{exit}$ are invariant for the respective chains $Y^A$, $Y^{\to A}$, and $Y^{A^c \to}$.

1) Let us prove surjectivity of the mapping $\mu \mapsto \mu_A$. Let $\nu$ be a locally finite non-zero Borel invariant measure of the induced chain $Y^A$ on $A$. As in \eqref{eq: Kac MC 1} applied to one-dimensional cylindrical sets $E=C_B$, we can ``lift'' $\nu$ from $A$ to a measure on~$\XX$:
\[
\bar \nu(B):= \E_\nu \!\left [ \sum_{k=0}^{\tau_A'(Y) - 1} \I(Y_k \in B ) \right ] = \int_A \E_y \!\left [ \sum_{k=0}^{\tau_A'(Y) - 1} \I(Y_k \in B ) \right ] \nu(dy), \qquad B \in \mathcal{B}(\XX).
\]
Note the difference with equality \eqref{eq: Kac} in Lemma~\ref{lem: induced3}, where we considered measures on the trajectory space $\XX^{\N_0}$ instead of measures on $\XX$ as here. The crucial observation is that $\tau_A'(Y)$ is finite $\P_\nu$-a.e., although we do not require that $\nu \ll \mu_A$ as in Lemma~\ref{lem: induced3}. In fact, we have $\P_y(\tau_A'(Y) \le \tau_{\Int(A)}'(Y)< \infty)=1$ for {\it every} $y \in \XX$. 

Then by the same argument as in \eqref{eq: Kac proof 1}, from the equality $\P_\nu(\tau_A'(Y) =\infty)=0$ we obtain
\[
\bar \nu(B) =  \sum_{k=0}^\infty \P_{\nu} (Y_k \in B, \tau_A'(Y)>k), \qquad B \in \mathcal B (\XX),
\]
and then
\begin{align} \label{eq: lifted invariant}
&\mathrel{\phantom{=}} \P_{\bar \nu}(Y_1 \in B) = \int_\XX \P_y(Y_1 \in B ) \bar \nu(dy) = \sum_{k=0}^\infty  \int_\XX \P_y(Y_1 \in B ) \P_\nu(Y_k \in dy, \tau_A'(Y)>k) \notag \\
&=\sum_{k=0}^\infty \P_\nu(Y_{k+1} \in B, \tau_A'(Y) \ge k+1) = \E_\nu \! \left [ \sum_{k=1}^{\tau_A'(Y)} \I(Y_k \in B) \right ] =\bar \nu (B),
\end{align}
where the second to the last inequality is again analogous to \eqref{eq: Kac proof 1} and in the last equality we used the relation $Y_{\tau_A'(Y)} = Y_1^A$ and the assumed invariance of $\nu$ for the induced chain $Y^A$.

Thus, $\bar \nu$ is an invariant measure for the chain $Y$ and it clearly satisfies $\bar \nu|_A = \nu$. Since $\nu$ is locally finite in the subspace topology, we can find a set $G \subset \Int(A)$ open in this topology such that $0<\nu(G)<\infty$. By Lemma~\ref{lem: Radon} this implies that $\bar \nu$ is locally finite since $G$ is also open in the topology of $\XX$. Thus, the mapping $\mu \mapsto \mu_A$ is surjective. 

2) We first consider the mapping $\mu \mapsto \mu_A^{entr}$. To prove its surjectivity, let $\nu$ be a locally finite non-zero Borel invariant measure of the entrance chain $Y^{\to A}$ on $A$. The Borel measure
\begin{equation} \label{eq: nu lifted}
\mu_1(B):= \E_\nu \! \left [ \sum_{k=0}^{T_1^{\to A} - 1} \I(Y_k \in B ) \right ], \qquad B \in \mathcal{B}(\XX),
\end{equation}
is invariant for $Y$, which can be checked exactly as in the proof of Part 1 above. In fact, we have $Y_1^{\to A}= Y_{T_1^{\to A}}$ by definition of the entrance chain, where $\P_\nu$-a.s.\ finiteness of $T_1^{\to A}$ follows from the strong Markov property of $Y$ combined with the equalities
\[
\P_y(\tau_A'(Y) \le \tau_{\Int(A)}'(Y)< \infty) = \P_y( \tau_{A^c}'(Y) \le \tau_{\Int(A^c)}'(Y)< \infty)=1, \qquad y \in \XX.
\]

Further, for any Borel set $B \subset A$, we have
\begin{align} \label{eq: nu from mu_1}
\int_{A^c} \P_y(Y_1 \in B ) \mu_1(dy) &= \sum_{k=0}^\infty  \int_{A^c} \P_y(Y_1 \in B ) \P_\nu(Y_k \in dy, T_1^{\to A} >k) \notag \\
&=\sum_{k=0}^\infty \P_\nu(Y_{k+1} \in B, T_1^{\to A} = k+1) \notag  \\
&= \P_\nu (Y_1^{\to A} \in B) = \nu(B).
\end{align}

By the assumption we have $\P_x(Y_1 \in \Int(A))>0$ for some $x \in \Int(A^c)$, and it follows that there exists a $z \in \Int(A)$ such that for any open set $U$ satisfying $z \in U \subset \Int(A)$ we have $\P_x(Y_1 \in U)>0$.  In fact, if this was not true, then every $z \in \Int(A)$ would admit an open neighbourhood $U_z \subset \Int(A)$ such that $\P_x(Y_1 \in U_z) = 0$. Hence the Borel measure $\P_x(Y_1 \in \cdot \cap \Int(A))$ is zero on compact sets, and by inner regularity of finite Borel measures on Polish spaces (Bogachev~\cite[Theorem~7.1.7]{Bogachev}), we arrive at $\P_x(Y_1 \in \Int(A)) =0$, which is a contradiction. 

By local finiteness of $\nu$, choose an open set $U$ satisfying $z \in U \subset \Int(A)$ such that $\nu(U)<\infty$. It holds $\P_x(Y_1 \in U)>0$ and by the weak Feller property of $Y$, we can find an open set $U_x$ such that $x \in U_x \subset \Int(A^c)$ and $\P_y(Y_1 \in U) \ge \frac12 \P_x(Y_1 \in U)$ for every $y \in U_x$. By \eqref{eq: nu from mu_1} and exactly the same argument as in \eqref{eq: finiteness}, this gives $\mu_1(U_x) < \infty$. Hence the measure $\mu_1$ on $\XX$ is locally finite by Lemma~\ref{lem: Radon}, and so the mapping $\mu \mapsto \mu_A^{entr}$ is surjective. Its injectivity follows immediately from Part 2a of Proposition~\ref{prop: Kac MC}.

Now consider the mapping $\mu \mapsto \mu_{A^c}^{exit}$. To prove its surjectivity, let $\nu^{exit}$ be a locally finite non-zero Borel invariant measure of the exit chain $Y^{A^c \to}$ on $A^c$. Then the Borel measure $\nu:=\int_{A^c} \P_y(Y_1 \in \cdot | Y_1 \in A) \nu^{exit}(dy)$ on $A$ is invariant for the entrance chain $Y^{\to A}$ from $A^c$ to $A$, and the measure $\mu_1$ introduced in \eqref{eq: nu lifted} is invariant for the chain $Y$. Moreover, we have the following equality of Borel measures on~$A^c$: 
\begin{align} \label{eq: nu_exit}
\P_y(Y_1 \in A) \mu_1(dy) &= \sum_{k=0}^\infty \P_y(Y_1 \in A ) \P_\nu(Y_k \in dy, T_1^{\to A} >k) \notag \\
&= \P_\nu(Y_{ T_1^{\to A} - 1} \in dy) = \P_\nu(Y_1^{A^c \to})=\nu^{exit}(dy), \qquad y \in A^c.
\end{align}

Then, if $x \in A^c$ is such that $\P_x(Y_1 \in \Int(A))>0$, by weak Fellerness of $Y$ and local finiteness of $\nu^{exit}$ we can choose an open set $U$ such that $x \in U \subset \Int(A^c)$, $\nu^{exit}(U)$ is finite, and $\P_y(Y_1 \in \Int(A)) \ge \frac12 \P_x(Y_1 \in \Int(A)) $ for every $y \in U$. By \eqref{eq: nu_exit}, this gives 
\[
\mu_1(U) = \int_U \frac{\nu^{exit}(dy)}{\P_y(Y_1 \in A) } \le \int_U \frac{\nu^{exit}(dy)}{\P_y(Y_1 \in \Int(A)) } \le \frac{2 \nu^{exit}(U)}{\P_x(Y_1 \in \Int(A)) } < \infty,
\]
hence the measure $\mu_1$ on $\XX$ is locally finite by Lemma~\ref{lem: Radon}. So the mapping $\mu \mapsto \mu_{A^c}^{exit}$ is surjective. Also, by the equality $\nu = \int_{A^c} \P_y(Y_1 \in \cdot) \mu_1(dy)$ of measures on $A$, $\nu$ is locally finite since $\mu_1$ is so, as we proved earlier. By the established injectivity of the mapping $\mu \mapsto \mu_A^{entr}$, this implies injectivity of the mapping $\mu \mapsto \mu_{A^c}^{exit}$. 
\end{proof}



\begin{proof}[{\bf Proof of Lemma~\ref{lem: weak Feller}}]
Consider the induced chain $Y^A$. Note that if the set $A$ is open, then this chain is well defined even without topological irreducibility of $Y$ since $Y$ returns to $A$ infinitely often by topological recurrence. We need to show that the mapping $x \mapsto \E_x f(Y^A_1)$ is continuous on $A$ for every continuous bounded function $f$ on $A$. Define the extension $\bar f$ of $f$ on $\XX$ by putting $\bar f:=0$ on $A^c$. Then for every $x \in A$,
\[
\E_x f(Y^A_1) = \sum_{k=1}^\infty \E_x \Bigl [ \bar f(Y_k) \I(Y_1, \ldots, Y_{k-1} \in A^c ) \Bigr ],
\]
and by the dominated convergence theorem and the fact that $\P_y(\tau_A'(Y)<\infty)=1$ for $y \in A$, it suffices to prove that every term is continuous on $\XX$. We use a simple inductive argument. It follows from the weak Feller property of $Y$ that the first term $\E_x \bar f(Y_1)$ is continuous at every $x \in \XX$ since the bounded function $\bar f$ is continuous on $(\partial A)^c$ and therefore continuous $\P_x(Y_1 \in \cdot)$-a.s. For every $k \ge 1$, by the Chapman--Kolmogorov equation we have
\[
\E_x \Bigl [ \bar f(Y_{k+1}) \I(Y_1, \ldots, Y_k \in A^c ) \Bigr ] = \int_\XX \I_{A^c}(y) \cdot  \E_y \Bigl [ \bar f(Y_k) \I(Y_1, \ldots, Y_{k-1} \in A^c ) \Bigr ] \P_x(Y_1 \in dy).
\]
Now we see that the above expression is a continuous function of $x$. This follows from the weak Feller property of $Y$ and $\P_x(Y_1 \in \cdot)$-a.s.\ continuity (in $y$) of the integrand, whose second factor is continuous by assumption of induction.

Similarly, the weak Feller property of the entrance chain $Y^{\to A}$ follows from the identity
\[
\E_x f(Y^{\to A}_1) = \sum_{n=1}^\infty \sum_{k=1}^\infty \E_x \Bigl [ \bar f(Y_{n+k}) \I(Y_1, \ldots, Y_{n-1} \in A, Y_n, \ldots, Y_{n+k-1} \in A^c ) \Bigr ]
\]
by induction on $n$ using the continuity $\E_x f(Y^A_1)$ obtained above for the basis step $n=1$. 
\end{proof}

\section{Applications to random walks in $\R^d$} \label{Sec: Application to RWs}

In this section we apply the ideas developed in Section~\ref{sec: stationarity MC} to random walks in arbitrary dimension. In particular, we answer our initial questions on stationarity properties of the chain of overshoots of a one-dimensional random walk over the zero level. Although the results of this section follow easily from those of Section~\ref{sec: stationarity MC}, we state them as separate theorems.

Recall that the state space $\ZZ$ of the random walk $S$ in $\R^d$, where $d\ge 1$, was defined in the Introduction as the minimal closed subgroup of $(\R^d, +)$ containing the topological support of the distribution of $X_1$. Let us normalize the Haar measure $\lambda$ on $\ZZ$ such that $\lambda([0,x))=\lambda'([0,x))$ for any $x \ge 0$ in $\ZZ$, where $\lambda'$ is the Lebesgue measure on the linear hull $\lin (\ZZ)$ of $\ZZ$  and $[0,x):=\{y \in \R^d: 0 \le y < x\}$; we always mean that inequalities between points in $\R^d$ hold coordinate-wise. In this section we assume w.l.o.g.\ that $\lin (\ZZ)$ has full dimension. We say that $S$ is {\it lattice} if $\ZZ = \Z^d$ and {\it non-lattice} is otherwise.

Clearly, $\lambda$ is invariant for the walk $S$ on $\XX=\ZZ$. We can say more.
\begin{lemma} \label{lem: RW ergodic}
Any topologically recurrent random walk $S$ on its state space $\ZZ$, where $\ZZ \subset \R^d$ and $d\in \{1,2\}$, is recurrent and ergodic starting under $\lambda$, which is the unique (up to multiplication by constant) locally finite Borel invariant measure of $S$ on $\ZZ$.
\end{lemma}

Recall that topological recurrence of the random walk $S$ on $\ZZ$ by definition means that $\P_0(S_n \in G \text{ i.o.})=1$ for every open neighbourhood $G \subset \ZZ$ of $0$. Note that for such random walks, this equality is in fact true for {\it every} non-empty open set $G \subset \ZZ$; see Revuz~\cite[Proposition~3.4]{Revuz}. 
Combined with the results of Chung and Fuchs~\cite[Theorems 1, 3 and~4]{ChungFuchs}, this gives that topological recurrence of $S$ on $\ZZ$ is equivalent  to 
\[
\limsup_{r \to 1-} \int_{[-a,a]^d} \frac{1}{\Re (1- r\E e^{i t \cdot X_1})}dt = \infty \quad \text{for all }a>0;
\]
the limit is always finite for $d \ge 3$. For $d=1$, the limit can be switched with the integral (Ornstein~\cite{Ornstein}). In particular, for $d=1$ this integral diverges when $\E X_1=0$, and it may also diverge for arbitrarily heavy-tailed $X_1$ (Shepp~\cite{Shepp}). In dimension $d=2$, $S$ is topologically recurrent on $\ZZ$ if $\E X_1=0$ and $\E \|X_1\|^2 <\infty$ (Chung and Lindvall~\cite{ChungLindvall}). For more general results on recurrence of random walks on locally compact abelian metrizable groups, see Revuz~\cite[Chapters 3.3 and 3.4]{Revuz}.

\begin{proof}
The uniqueness is by Proposition~I.45 in Guivarc'h et al.~\cite{French}, which states that the right Haar measure on a locally compact Hausdorff topological group $G$ with countable base is a unique invariant Radon Borel measure for any topologically recurrent right random walk on $G$ such that no proper closed subgroup of $G$ contains the support of the distribution of increments of the walk. 

To infer ergodicity, note that uniqueness of invariant measure implies irreducibility of $S$ starting under $\lambda$. In fact, if there is a $\lambda$-non-trivial invariant set $A \in \mathcal{B}(\ZZ)$ of $S$, then the locally finite measure $\I_A \lambda$ is invariant for $S$, which contradicts the uniqueness. Further, by Revuz~\cite[Proposition~3.4]{Revuz}, topological recurrence of $S$ implies that $\P_x(\tau_G'(S) < \infty)=1$ for every $x \in \ZZ$ and every non-empty open set $G \subset \ZZ$. Hence $S$ is recurrent starting under $\lambda$ by Condition~\ref{cond: recur finite} in Section~\ref{sec: setup MCs}. Therefore, $S$ is ergodic by irreducibility and recurrence, all the three properties starting under~$\lambda$ (Kaimanovich~\cite[Proposition~1.7]{Kaimanovich}).
\end{proof}

We say that a Borel set $A \subset \ZZ$ is {\it massive} for the random walk $S$ if $\P_x(\tau_A'(S)<\infty)=1$ for $\lambda$-a.e.\ $x \in \ZZ$. In particular, if $S$ is topologically recurrent, then any Borel set of positive measure $\lambda$ is massive, as follows (Aaronson~\cite[Proposition~1.2.2]{Aaronson}) from ergodicity and recurrence of $S$ starting under $\lambda$ (Lemma~\ref{lem: RW ergodic}). If $S$ is transient (i.e.\ not topologically recurrent), no set of finite measure can be massive. For walks on $\ZZ=\Z^d$ with $d \ge 3$  satisfying $\E X_1 =0$ and $\E \|X_1\|^2 <\infty$, there is a necessary and sufficient condition for massiveness of a set, called Wiener's test, stated in terms of capacity, by It\^o and McKean~\cite{ItoMcKean} and Uchiyama~\cite{Uchiyama}. Easily verifiable sufficient conditions for massiveness in $d=3$ are due to Doney~\cite{Doney1965}. For example, any ``line'' in $\Z^3$ is massive. Under the above assumptions, a set is massive for every such a walk if it is massive for a simple random walk, and so this is a property of a set rather than of a walk. Apart from partial results of Greenwood and Shaked~\cite{GreenwoodShaked} (mentioned below) for convex cones with apex at the origin, we are not aware of any explicit results for non-lattice random walks. It appears (based on the estimates of Green's function in Uchiyama~\cite[Section~8]{UchiyamaGreen}) that such results should be fully analogous to the lattice ones for walks with $\E X_1 =0$ and $\E \|X_1\|^2 <\infty$ if the distribution of $X_1$ has density with respect to the Lebesgue measure. The case of heavy-tailed random walks on $\Z^d$, including transient walks in dimensions $d=1,2$, is considered by Bendikov and Cygan~\cite{BendikovCygan1, BendikovCygan2}. 

Since $-A$ is massive for $S$ if and only if $A$ is massive for $-S$, and the random walk $-S$ is dual to $S$ with respect to the measure $\lambda$ (see, e.g.\ equality (2.24) in~\cite{MijatovicVysotsky}), Theorem~\ref{thm: inducing MC} immediately implies the following result.

\begin{theorem} \label{thm: RW general}
Assume that the sets $A$, $-A$, $A^c$, $-A^c$ are massive for a random walk $S$ on its state space $\ZZ$, where $\ZZ \subset \R^d$ and $d \ge 1$. Then the measures $\lambda_A^{entr}(dx) = \P(X_1 \in x - A^c) \lambda(dx)$ on $A$ and $\lambda_{A^c}^{exit}(dx) = \P(X_1 \in A - x) \lambda(dx)$ on $A^c$ are invariant for the entrance chain $S^{\to A}$ and exit chain $S^{A^c \to}$, respectively.
\end{theorem}

\begin{remark}
If $\ZZ=\Z^d$ with $d \ge 3$, $\E X_1 =0$ and $\E \|X_1\|^2 <\infty$, then the assumptions on $-A$ and $-A^c$ in Theorem~\ref{thm: RW general} are not required since a set is massive for $S$ whenever it is massive for a simple random walk, which is self-dual. We do not know if such reduction is possible for arbitrary $S$.
\end{remark}

Let us discuss two particular cases. First, if the random walk $S$ is topologically recurrent on $\ZZ$, then the assumptions of Theorem~\ref{thm: RW general} are satisfied for any $\lambda$-non-trivial Borel set $A$. Second, $A$ is of the form $A=A'\cap \ZZ$, where $A'$ is a convex cone in $\R^d$ with apex at zero (and $S$ may be transient). Here massiveness of $A^c$ and $-A^c$ follows from that of $A$ and $-A$. In the case of the non-negative orthant $A'=[0, \infty)^d$, which is of special interest, we have
\[
\{X_1 \in x - A^c\} = \{X_1 \in (x - A)^c\} = \{X_1 \not \in x - A\} = \{X_1 \not \le x\} \text{ a.s.},
\]
where in the last expression and below we mean that inequalities between points in $\R^d$ hold coordinate-wise. Combining this with the analogous expression for the negative orthant, we get the following.

\begin{cor} \label{cor: orthant}
Assume that $\tau_\pm:=\tau_{\pm(0, \infty)^d}'(S)$ are finite $\P_0$-a.s. Then the measures 
$$
\lambda_{[0, \infty)^d}^{entr} = \I_{[0,\infty)^d} (x) (1 - \P(X_1 \le x)) \lambda(dx) \quad \text{and} \quad 
\lambda_{(-\infty,0)^d}^{entr} = \I_{(-\infty,0)^d} (x) (1 - \P(X_1 > x)) \lambda(dx)
$$ 
are invariant for the chains of entrances into $[0, \infty)^d$ and $(-\infty, 0)^d$, respectively.
\end{cor}

Note that the measures $\lambda_{[0, \infty)^d}^{entr}$ and $\lambda_{(-\infty,0)^d}^{entr}$ are always infinite if $\dim(\lin(\ZZ)) \ge 2$.

It is clear that the assumptions of the corollary imply that the expectation of every coordinate $X_1^{(k)}$ of $X_1=(X_1^{(1)}, \ldots, X_1^{(d)})$ is either $0$ or does not exist, i.e.\ $\E (X_1^{(k)})^+ = \E (X_1^{(k)})^- =+\infty$, where $x^+:= \max\{x, 0\}$ and $x^-:=(-x)^+$ for a real $x$. In dimension one, where $\tau_+$ and $\tau_-$ are the first ascending and descending ladder times, this is actually an equivalence (cf.~Feller~\cite[Theorem~XII.2.1]{Feller} and Kesten~\cite[Corollary~3]{Kesten}). This is also equivalent to assuming that the one-dimensional random walk $S$ {\it oscillates}, that is $\limsup S_n = -\liminf S_n = +\infty$ a.s.\ as $n \to \infty$. We are not aware of necessary and sufficient conditions for $\P_0$-a.s.\ finiteness of $\tau_+$ and $\tau_-$  in higher dimensions. By Greenwood and Shaked~\cite[Corollary~3]{GreenwoodShaked}, in any dimension a sufficient condition is 
\[
\sum_{n=1}^\infty \frac1n \P_0(S_n >0)=\sum_{n=1}^\infty \frac1n \P_0(S_n <0) = +\infty.
\]

We now state a uniqueness result. 

\begin{theorem} \label{thm: uniqueness}
Let $S$ be any topologically recurrent random walk on its state space $\ZZ$, where $\ZZ \subset \R^d$ and $d \in \{1,2\}$, and let $A \subset \ZZ$ be any $\lambda$-non-trivial Borel set with $\lambda(\partial A)=0$. Then the entrance chain $S^{\to A}$ is ergodic and recurrent starting under $\lambda_A^{entr}$, which is the unique (up to multiplication by constant) locally finite Borel invariant measure of $S^{\to A}$. The same is true for $S^{A^c \to}$ starting under $\lambda_{A^c}^{exit}$.
\end{theorem}

This follows from Theorems~\ref{thm: inducing MC},~\ref{thm: inducing bijection} and Lemma~\ref{lem: RW ergodic} using that the transition kernel of any random walk is weak Feller by $\P_x(S_1 \in \cdot)=\P(x+ X_1 \in \cdot)$; the inequality $\P_{\lambda_{\Int(A^c)}}(S_1 \in \Int(A))>0$ holds true since otherwise $\Cl(A)$  would be a $\lambda$-non-trivial set invariant for $S$. The assumption  $\lambda(\partial A)=0$ of the theorem can be relaxed but we prefer to avoid considering sets with ``thick'' boundary such as $(\R \setminus \Q) \cap [0,1]$.

Recall that  for $d=1$ we have $\pi= \frac12 \pi_+ + \frac12 \pi_-$ with  $\pi_+=c_1 \lambda_{[0, \infty)}^{entr}$ and $\pi_- = c_1 \lambda_{(- \infty,0)}^{entr} $.

\begin{cor} \label{cor: overshoots}
If a one-dimensional random walk $S$ is topologically recurrent on $\ZZ$, then the chains of overshoots $O$, $O^\downarrow$, and $\mathcal{O}$ (defined in the Introduction) are ergodic and recurrent starting respectively under their unique normalized invariant measures $\pi_+$, $\pi_-$, and $\pi$.
\end{cor}

Finally, we comment on stability of the ``distribution'' of the entrance chain into $A$. This question makes a probabilistic sense only if the measure $\lambda_A^{entr}$ is finite and therefore can be normalized to be a probability. For example, this is the case when $A = [0, \infty) $, $\E X_1 =0$ and $d=1$ or the $S$ is topologically recurrent on $\ZZ$, $A$ is bounded, and $d \in \{1,2\}$. In the former case, the question of stability is studied in our companion paper~\cite{MijatovicVysotsky}. In the latter case, it is reasonable to restrict the attention to convex and compact sets $A$. These are intervals for $d=1$, considered in~\cite[Section~5.1]{MijatovicVysotsky}. It appears that convergence results in dimension $d=2$ can be obtained using exactly the same approach as in~\cite{MijatovicVysotsky}.


\section{Futher results on level-crossings for one-dimensional random walks} \label{Sec: L_n}

Throughout  this section the random walk $S$ is  one-dimensional.

\subsection{The limit theorem for the number of level-crossings} 
Recall that $L_n$, defined in \eqref{eq: Ln def}, denotes the number of zero-level crossings of  $S$ by time~$n$. Combining Theorem~\ref{thm: uniqueness} on ergodicity of the chain of overshoots with a result by Perkins~\cite{Perkins} on convergence of local times of random walks, we obtain the following central limit theorem for $L_n$. To the best of our knowledge, all other results of this type require some smoothness assumptions for the distribution of increments of the walk. 

\begin{theorem} \label{thm: level-crossings}
For any random walk $S$ such that $\E X_1 =0$ and $\sigma^2:=\E X_1^2 \in (0, \infty)$, we have
$$
\lim_{n \to \infty} \P_x \bigg( \frac{L_n}{\sqrt{n}} \le y \bigg) = 2\Phi \bigg( \frac{\sigma y}{2\E|X_1|}  \bigg)-1, \qquad x \in \ZZ, \, y \ge 0,
$$ 
where $\Phi$ denotes the distribution function of a standard normal random variable.
\end{theorem}


We will need the following auxiliary result, the law of large numbers for the chain $\mathcal O$. It does not follow directly from ergodicity of $\mathcal O$ (stated in Corollary~\ref{cor: overshoots}) since Birkhoff's ergodic theorem implies convergence of the time averages only for $\pi$-a.e.\ $x$.

\begin{proposition} \label{prop: LLN}
Let $S$ be any random walk such that $\E X_1 =0$ and $\sigma^2:=\E X_1^2 \in (0, \infty)$. Then for \emph{every} $x \in \ZZ$,
\begin{equation} \label{eq: LLN overshoots}
\lim_{n \to \infty} \frac{1}{n} \sum_{k=1}^n |\mathcal{O}_k| = \int_{\ZZ} |y| \pi(dy) = \frac{\sigma^2}{2 \E |X_1|}, \quad \P_x \text{-a.s.}
\end{equation}
\end{proposition}

\begin{proof}[{\bf Proof of Theorem~\ref{thm: level-crossings}}]
Denote by $\ell_0$ the local time at $0$ at time $1$ of a standard Brownian motion. By L\'evy's theorem, $\ell_0$ has the same distribution as the absolute value of a standard normal random variable. Combining this result with Theorem~1.3  by Perkins~\cite{Perkins}, we get
\begin{equation} \label{eq: Perkins}
\lim_{n \to \infty} \P_x \bigg (\frac{1}{\sigma \sqrt{n}} \sum_{k=1}^{L_n} |\mathcal{O}_k|  \le y \bigg ) = 2\Phi (y)-1, \qquad x=0, \, y \ge 0;
\end{equation}
note that since Perkins's definition of crossing times is slightly different from the one of ours, his result shall be applied to the random walk $-S/\sigma$. On the other hand, by Proposition~\ref{prop: LLN}, 
\begin{equation} \label{eq: LLN overshoots 2}
\lim_{n \to \infty} \frac{1}{L_n'} \sum_{k=1}^{L_n} |\mathcal{O}_k| = \frac{\sigma^2}{2 \E |X_1|}, \quad \P_x \text{-a.s.}, \qquad x \in \ZZ,
\end{equation}
where $L_n':=L_n+\I(L_n = 0)$ and we used the fact that $\P_x(\lim_{n \to \infty} L_n = \infty)=1$, which holds true since $S$ oscillates. Rewriting  equality \eqref{eq: Perkins} using the identity $\frac{1}{\sqrt n} = \frac{L_n'}{\sqrt n} \cdot \frac{1}{L_n'}$ and then combining it with \eqref{eq: LLN overshoots 2} yields the assertion of  Theorem~\ref{thm: level-crossings} for $x=0$ by Slutsky's theorem.

Furthermore, the results of Perkins actually imply (by Perkins~\cite{PerkinsPrivate}) that equality \eqref{eq: Perkins} remains valid, although this is not stated in~\cite[Theorem 1.3]{Perkins}, if we replace $x = 0$ by $x_n \in \ZZ$ for any sequence $(x_n)_{n \ge 1} \subset \ZZ$ such that $\lim_{n \to \infty} x_n/\sqrt n =0$. In particular, we can take $x_n \equiv x$ for an arbitrary $x \in \ZZ$, which yields Theorem~\ref{thm: level-crossings} in full by the above argument.

Let us explain in detail the extension of \eqref{eq: Perkins} stated above. For $x=0$, Theorem 1.3 of  Perkins~\cite{Perkins}  is an immediate corollary to his Lemma 3.2 and Corollary 2.2. Our extension of \eqref{eq: Perkins}  follows in exactly the same way if we let $x$ in Lemma 3.2 be the {\it nearstandard} point in $^*\R$, the field of {\it nonstandard real numbers}, that corresponds to the sequence $(x_n)_{n \ge 1} $, in which case $st(x) = \text{\textdegree} x = 0$, i.e.\ the {\it standard part} of $x$ is $0$. We referred to Cutland~\cite{Nonstandard} to digest the  unusual notation and concepts of nonstandard analysis, which were used in~\cite{Perkins} with no explanation.
\end{proof}

\begin{proof}[{\bf Proof of Proposition~\ref{prop: LLN}}]


Denote $h:=\inf\{z \in \ZZ: z>0\}$; then either $\ZZ=h\Z$ if $h>0$ or $\ZZ=\R$ if $h=0$.
One can easily check that for $\pi_+$ (defined in \eqref{eq: inv quadrant} with $c_1=2/\E|X_1|$),
\begin{equation} \label{eq: E 1}
\int_{\ZZ} y \pi_+(dy) = \frac{2}{\E |X_1|} \int_h^\infty (y -h/2) \P(X_1 >y) dy = \frac{2}{\E |X_1|}\int_0^\infty (y -h/2) \P(X_1 >y) dy 
\end{equation}
and, similarly, 
\begin{equation} \label{eq: E 2}
-\int_{\ZZ} y \pi_-(dy) = \frac{2}{\E |X_1|} \int_0^\infty (y +h/2) \P(-X_1 >y) dy.
\end{equation}
Using that $\E X_1 = 0$ and integrating the above equality by parts, we find that the probability measure $\pi$, which, recall, satisfies $\pi = \frac12 \pi_+ + \frac12 \pi_-$, has the first absolute moment $\sigma^2/(2 \E|X_1|)$. Therefore, by Birkhoff's ergodic theorem and ergodicity of the chain of overshoots $\mathcal O$ asserted in Corollary~\ref{cor: overshoots}, the convergence in \eqref{eq: LLN overshoots} holds true for $\pi$-a.e.\ $x \in \ZZ$. We need to prove this for every $x \in \ZZ$. 

Denote by $\supp \pi$ the topological support of $\pi$ and by $N$ the set of points $x \in \supp \pi$ that satisfy \eqref{eq: LLN overshoots}. We clearly have $N = \supp \pi$ in the lattice case $h>0$, where $\ZZ$ is discrete. In the non-lattice case $h=0$, so far we only have that $N$ is dense in $\supp \pi$. This is because $N$ has full measure $\pi$, hence $N$ has full Lebesgue measure $\lambda|_{\supp \pi}$, as readily seen from definition~\eqref{eq: pi} of $\pi$. In order to prove \eqref{eq: LLN overshoots}, we need to show that $N = \supp \pi$, since the chain $\mathcal{O}$ hits the support of $\pi$ (which is a closed interval, possibly infinite) at the first step regardless of the starting point. Our argument goes as follows.

Consider the random walk $S':= (S_n')_{n \ge 0}$, where $S_n'= X_1+ \ldots + X_n$ for $n \ge 1$,  starting at $S_0':=0$. Then $\P_x(S \in \cdot) = \P\big((x+S_0', x+ S_1', \ldots) \in \cdot\big)$. For real $y_1, y_2$, define the functions
\[
g(y_1, y_2):= \I(y_1<0, y_2 \ge 0 \text{ or } y_1 \ge 0, y_2< 0), \quad f(y_1, y_2) := |y_2| g(y_1, y_2).
\]
We claim that for any $x \in \supp \pi$ and $\varepsilon \in (0,1)$, there exists a $ y \in N$ such that 
\begin{equation} \label{eq: approximation by N}
\limsup_{n \to \infty} \left | \frac{\sum_{k=1}^n f(y + S_{k-1}', y+ S_k')}{\sum_{k=1}^n g(y + S_{k-1}', y+ S_k')} -  \frac{\sum_{k=1}^n f(x + S_{k-1}', x+ S_k')}{\sum_{k=1}^n g(x + S_{k-1}', x+ S_k')} \right | \le \varepsilon, \quad \P \text{-a.s.}
\end{equation}
This will imply that $x \in N$ and hence prove Proposition~\ref{prop: LLN}, since 
\begin{equation*} 
\P \left( \lim_{n \to \infty}  \frac{\sum_{k=1}^n f(y + S_{k-1}', y + S_k')}{\sum_{k=1}^n g(y + S_{k-1}', y + S_k')} = \frac{\sigma^2}{2 \E |X_1|} \right) = \P_y \left( \lim_{n \to \infty}  \frac{1}{L_n'} \sum_{k=1}^{L_n}  |\mathcal{O}_k| = \frac{\sigma^2}{2 \E |X_1|} \right) = 1,
\end{equation*}
where $L_n'=L_n+\I(L_n = 0)$ and  the last equality holds by definition of the set $N$ and the fact that $\P_y(\lim_{n \to \infty} L_n = \infty)=1$, which is true because $S$ oscillates. Thus, it remains to prove inequality \eqref{eq: approximation by N}.

From the identity $\frac{a_1}{b_1} - \frac{a_2}{b_2}= \frac{a_1}{b_1} \bigl (1 -  \frac{a_2}{a_1} \cdot \frac{b_1}{b_2}  \bigr)$ for $a_1, a_2, b_1, b_2 >0$ and the inequality $\big| 1 -\frac{a}{b} \big| < 2|a-1| + 2|b-1|$ for $a>0, b>\frac12$, we see that \eqref{eq: approximation by N} will follow if we show that for any $x \in \supp \pi$ and $\varepsilon \in (0, \sigma^2/ (2\E|X_1|))$, there exists a $ y \in N$ such that $\P$-a.s.,
\begin{equation} \label{eq: bound}
\limsup_{n \to \infty} \left [ \left | \frac{\sum_{k=1}^n f(x+S_{k-1}', x + S_k')}{\sum_{k=1}^n f(y+S_{k-1}', y + S_k')} -1 \right| + \left | \frac{\sum_{k=1}^n g(x+S_{k-1}', x + S_k')}{\sum_{k=1}^n g(y+S_{k-1}', y + S_k')} -1 \right| \right ] \le \frac{ \varepsilon \E|X_1|}{\sigma^2}.
\end{equation}

For any $\delta >0$, integer $k \ge 1$, and any $y \in N$ such that $|x- y| \le \delta$, we have
\[
|g(x + S_{k-1}', x+ S_k') - g(y + S_{k-1}', y+ S_k')| \le \I(|y+ S_{k-1}'| \le \delta \text{ or } |y+ S_k'| \le \delta)
\]
and
\begin{multline*}
|f(x + S_{k-1}', x+ S_k') - f(y + S_{k-1}', y+ S_k')| \\
\le \delta g(y + S_{k-1}', y + S_k') + (|y + S_k'| + \delta) \I(|y+ S_{k-1}'| \le \delta \text{ or } |y+ S_k'| \le \delta).
\end{multline*}
This gives
\begin{equation} \label{eq: estimate g}
\left | \frac{\sum_{k=1}^n g(x+S_{k-1}', x + S_k')}{\sum_{k=1}^n g(y+S_{k-1}', y + S_k')} -1 \right| \\
\le \frac{\sum_{k=1}^n \bigl[ \I(|y+ S_{k-1}'| \le \delta) +  \I( |y+ S_k'| \le \delta) \bigr ]}{\sum_{k=1}^n g(y+S_{k-1}', y + S_k')}
\end{equation}
and
\begin{multline} \label{eq: estimate f}
\left | \frac{\sum_{k=1}^n f(x+S_{k-1}', x + S_k')}{\sum_{k=1}^n f(y+S_{k-1}', y + S_k')} -1 \right| \\
\le \frac{\sum_{k=1}^n \bigl[ \delta g(y + S_{k-1}', y + S_k') + (|X_k| + 2\delta) \I(|y+ S_{k-1}'| \le \delta) + 2 \delta \I( |y+ S_k'| \le \delta) \bigr ]}{\sum_{k=1}^n f(y+S_{k-1}', y + S_k')}.
\end{multline}

By Lemma~\ref{lem: Radon}, the topologically recurrent random walk $S$ on $\ZZ=\R$ is recurrent and ergodic starting under the Lebesgue measure $\lambda$. By Condition~\ref{cond: conserv} in Section~\ref{sec: setup MCs}, recurrence of $S$ starting under $\lambda$ implies conservativity of the measure preserving one-sided shift $\theta$ on $(\R^{\N_0}, \mathcal{B}(\R^{\N_0}), \PP_{\lambda}^S)$. Therefore we can apply Hopf's ratio ergodic theorem (see the Appendix) to the ratios on the r.h.s.'s of \eqref{eq: estimate g} and \eqref{eq: estimate f}. Let us explain in details, say, why 
\begin{equation}\label{eq:Hopfpart}
\P \left (\lim_{n\to \infty}\frac{\sum_{k=1}^n  g(y + S_{k-1}', y + S_k') }{\sum_{k=1}^n f(y+S_{k-1}', y + S_k')} = \frac{\E|X_1|}{\sigma^2/2}\right) =1, \qquad \lambda\text{-a.e. } y.
\end{equation}

Indeed, consider the functions on $\R^{\N_0}$ defined by $G(z):=g(z_0, z_1)$ and $F(z):=f(z_0,z_1)$ for $z=(z_0, z_1, \ldots) \in \R^{\N_0}$. Both functions are non-negative, non-zero, and $\PP_{\lambda}^S$-integrable by
\[
\EE_{\lambda}^S  G= \int_\R \E_{z_0} g(S_0, S_1)  \lambda(d z_0) 
=  \int_{-\infty}^0 \P(z_0+X_1 \ge 0) dz_0 + \int_0^\infty \P(z_0+X_1 <0) d z_0  = \E|X_1|
\]
and 
\begin{align*}
\EE_{\lambda}^S F&= \int_\R \E_{z_0} [|S_1| g(S_0, S_1)]  \lambda(d z_0) \\
&=  \int_{-\infty}^0 \E [(z_0 + X_1) \I (z_0+X_1 \ge 0)] dz_0 - \int_0^\infty \E [(z_0 + X_1) \I (z_0+X_1 < 0)] d z_0 \\
&=  \int_0^\infty \E [(| X_1| -z_0) \I (|X_1|  > z_0)] d z_0 = \E |X_1|^2/2,
\end{align*}
where the last equality follows from Fubini's theorem. Finally, we have
\begin{multline*}
\PP_{\lambda}^S \left (\limsup_{n\to \infty} \bigg| \frac{\sum_{k=0}^{n-1} G\circ \theta^k}{\sum_{k=0}^{n-1} F \circ \theta^k}  - \frac{\EE_{\lambda}^S G}{\EE_{\lambda}^S F}  \bigg| \neq 0  \right) \\
=
\int_\R \P \left (\limsup_{n\to \infty} \bigg| \frac{\sum_{k=1}^n  g(y + S_{k-1}', y + S_k') }{\sum_{k=1}^n f(y+S_{k-1}', y + S_k')} - \frac{\E|X_1|}{\sigma^2/2}\bigg| \neq 0  \right) \lambda(dy),
\end{multline*}
hence equality~\eqref{eq:Hopfpart} follows from Hopf's ratio ergodic theorem.

Similarly to~\eqref{eq:Hopfpart}, for every $\delta >0$, for $\lambda$-a.e.\ $y$ the sum of the ratios on the r.h.s.'s of \eqref{eq: estimate g} and \eqref{eq: estimate f} converges $\P$-a.s.\ as $n \to \infty$ to 
\begin{equation*} 
c(\delta):=\frac{\delta \E |X_1| + 2 \delta (\E |X_1| + 2\delta) + 4 \delta^2 }{\sigma^2/2} + \frac{4 \delta}{\E |X_1|}.
\end{equation*}
Denote by $N_\delta$ the set of $y$ where this $\P$-a.s.\ convergence holds true. Choose a $\delta>0$ such that $c(\delta)<\varepsilon \E|X_1|/ \sigma^2$. The Borel set $N \cap N_\delta$ has full measure $\lambda|_{\supp \pi}$ and hence is dense in $\supp \pi$. Therefore we can pick a $y \in N \cap N_\delta$ that satisfies $|x-y| \le \delta$. Then inequality~\eqref{eq: bound} follows from \eqref{eq: estimate g} and \eqref{eq: estimate f}, as required. 
\end{proof}

\subsection{Expected occupation times between level-crossings}
In the rest of the section we present several identities for occupation times, which are direct corollaries of Proposition~\ref{prop: Kac MC} on general Markov chains. 

Define the first up- and down-crossing times of the zero level as 
\[
 T:= \inf\{k\ge 1: S_{k-1} < 0, S_k \ge 0 \}  \quad \text{and} \quad
 T^\downarrow:= \inf\{k\ge 1: S_{k-1} \ge 0, S_k < 0 \}.
\]
Recall that $\mathcal{T}_1 = \min\{T, T^\downarrow\}$, $\pi_+=c_1 \lambda_{[0, \infty)}^{entr}$ and $\pi_- = c_1 \lambda_{(- \infty,0)}^{entr} $ (see~\eqref{eq: inv quadrant}).

\begin{proposition} \label{prop: reconstruct}
For any random walk $S$ that oscillates, for any Borel set $B \subset \ZZ$ we have
\begin{equation*} 
c_1 \lambda(B) = \E_{\pi_+} \!\!\left [ \sum_{k=0}^{T-1} \I(S_k \in B) \right ] = \E_{\pi_-} \!\!\left [ \sum_{k=0}^{T^\downarrow-1} \I(S_k \in B) \right ] = 2 \E_{\pi} \!\left [ \sum_{k=0}^{\mathcal{T}_1-1} \I(S_k \in B) \right] \!.
\end{equation*}
\end{proposition}
We have not seen these formulas in the random walks literature with exception of one particular case when $S$ is a symmetric simple random walk. Here $\pi_+=\delta_0$ and the first formula above is described by Feller~\cite[Section XII.2, Example~b]{Feller}, who praised ``the fantastic nature of this result''. In this case, the above formulas are similar to the identity
\[
1 = \E_0 \!\left [ \sum_{k=0}^{\tau'_{\{0\}}(S) - 1} \I(S_k = x) \right ], \qquad x \in \ZZ,
\]
which holds for any lattice recurrent random walk $S$ and corresponds to $A= \{0\}$ and $E = C_{\{x\}}$ in~\eqref{eq: Kac MC 1}. By the same reasoning, we get a version of this result for non-lattice recurrent walks, proved by Ornstein~\cite[Theorem~0.5]{Ornstein}: for any Borel set $A \subset \R$ of positive Lebesgue measure $\lambda$,
\[
\lambda(B) = \int_A \E_x \bigg [ \sum_{k=0}^{\tau'_A(S) - 1} \I(S_k \in B) \bigg ] \lambda(dx), \qquad B \in \mathcal{B}(\R).
\]
We stress again that in Proposition~\ref{prop: reconstruct} we do not assume recurrence of $S$.

Note also that the formulas in Proposition~\ref{prop: reconstruct} are similar to the pre-zero-crossing occupation measure representation for the renewal measure of ladder heights of a random walk starting at zero (Asmussen~\cite[Theorem~VIII.2.3b]{Asmussen}), but in our case the walk starts differently.  

\begin{proof}
The Markov chain $S$ on $\ZZ$, its invariant measure $c_1 \lambda$, and the set $\ZZ \cap [0, \infty)$ satisfy the assumptions of Part 2b of Proposition~\ref{prop: Kac MC}; see Section~\ref{Sec: Application to RWs}. Then the first equality follows from  the fact that $\pi_+=c_1 \lambda^{entr}_{ [0, \infty)} $ and formula \eqref{eq: Kac MC 2} applied to $E = C_B$. The second equality is analogous. The third one follows by applying the first two to the sets $B \cap  [0, \infty)$ and $B \cap (-\infty,0)$.
\end{proof}

Further, for the {\it number of up-crossings} of arbitrary level $a$ by time $n \ge 1$, defined as
\[
L_n^\uparrow(a):=\sum_{i=0}^{n-1} \I(S_i <a, S_{i+1} \ge a), \qquad a \in \ZZ,
\]
we obtain the following surprising result.

\begin{proposition} \label{prop: E up-crossings}
For any non-degenerate random walk $S$ satisfying $\E X_1=0$, we have $\E_{\pi_+} L_T^\uparrow(a) = 1$ and $\E_{\pi_-} L_T^\uparrow(a) = 1$ for any $a \in \ZZ$.
\end{proposition}

Thus, the expected number of up-crossings by the  time $T$ does not depend on the level (if $S$ is started under $\pi_+$ or $\pi_-$, i.e.\ at stationarity of either chain $O$ or $O^\downarrow$), and therefore equals~$1$ since $L_T^\uparrow(0)=1$ by the definition of $T$.

\begin{proof}
For the first equality, take $E=\{x \in \ZZ^{\N_0}: x_0 < a, x_1 \ge a \}$ and $A= \ZZ \cap [0, \infty)$ in \eqref{eq: Kac MC 2}, and use the facts that $\pi_+=c_1 \lambda^{entr}_{[0, \infty)} $ and $\P_{\lambda}(S \in E_a)=\P_{\lambda}(S \in E_0) =c_1^{-1}$, where the first equality follows by shift-invariance of both the measure $\lambda$ and the transition kernel of $S$. The second equality is analogous.
\end{proof}

From the idea that it is more natural to start the random walk from $0$ rather than under $\pi_+$, we can use Proposition~\ref{prop: E up-crossings} to find $\E_0 L_T^\uparrow(a)$ for two specific types of distributions of increments. We say that $X_1$ has {\it upward exponential distribution} if the conditional distribution $\P(X_1> \cdot | X_1>0)$ is exponential. For every distribution of this type we have $\pi_+(\cdot) = \P(X_1 \in \cdot | X_1>0)$, which by Proposition~\ref{prop: E up-crossings} and the memoryless property of exponential distributions easily implies (we omit the computations) that
\[
\E_0 L_T^\uparrow(a) = \frac{\P(X_1 >0)}{\P(X_1 \neq 0)} + \P(X_1 \ge a | X_1 >0) , \qquad a > 0.
\]
We say that $X_1$ has {\it upward skip-free distribution} if $\P(X_1 \in \{1, 0, -1, \ldots \}) = 1$. If the random walk $S$ has such increments, then $\pi_+=\delta_0$ and thus $\E_0 L_T^\uparrow(a) =1$ for every real $a$.

The main application of random walks with upward exponential distributions is in queuing theory, where they feature in the Lindley formula for the waiting times in GI/M/1 queues with exponential service times; see Asmussen~\cite[Section~III.6]{Asmussen}. The main application of random walks with skip-free distributions is in theory of branching processes; they also appear in queuing theory~\cite[Section~III.6]{Asmussen}.

\section*{Acknowledgements}
We thank Wolfgang Woess for discussions and Vadim Kaimanovich for providing a reference to his extremely useful paper. We thank the anonymous referees for their comments and suggestions.

\appendix
\section{Induced transformations in infinite ergodic theory} 

Here we present basic results on inducing for measure preserving transformations of infinite measure spaces; see Aaronson~\cite[Chapter 1]{Aaronson} for the introduction to infinite ergodic theory. We present few variations of Kakutani's classical results of 1943, mainly to cover inducing on sets of infinite measure. To our surprise, we did not find references for the exact statements we need, and we prove them here.

Let $T$ be a transformation of some measurable space $(X, \mathcal{F})$. For any set $A \in \mathcal{F}$, consider the {\it first hitting time} $\tau_A$ of $A$ and the {\it first hitting mapping} $\varphi_A$ defined by
\begin{equation} \label{eq: tau phi def}
\tau_A(x):=\inf\{n \ge 1: T^n x \in A\}, \,  x \in X \quad \text{and} \quad \varphi_A(x):=T^{\tau_A(x)} x, \, x \in \{\tau_A < \infty\}.
\end{equation}
and the {\it first return} or {\it induced mapping} $T_A:={(\varphi_A)|}_A$ defined on $A \cap \{\tau_A < \infty\}$. Put also
$\tilde \tau_A(x):=\inf\{n \ge 0: T^n x \in A\}$ and $\tilde \varphi_A(x):=T^{\tilde \tau_A(x)} x$. All these mappings are measurable.

From now on 
we assume that $m$ is a measure on $(X, \mathcal{F})$ and the transformation $T$ is measure preserving on $(X, \mathcal{F}, m)$. We say that a set $A \in \mathcal F$ is {\it recurrent} for $T$ if $\tau_A$ is finite $m$-a.e.\ on $A$, that is $A \subset \cup_{k \ge 1} T^{-k} A \Mod{m}$, where$\Mod{m}$ means true possibly except for a $m$-zero set. If $A$ is recurrent for $T$, then  from invariance of $m$ it follows by simple induction that all iterations of the mapping  $\varphi_A$ are defined $m$-a.e.\ on~$A$ (see~\cite[Section~1.5]{Aaronson}), that is
\begin{equation} 
\label{eq: io}
{m|}_A(\tau_A = \infty)=0 \quad \Longrightarrow \quad {m|}_A(\{T^k \in A \text{ i.o.}\}^c)=0,
\end{equation}
where ``i.o.'' stands for ``infinitely often''.

The following result on induced transformations essentially is in~\cite[Proposition~1.5.3]{Aaronson}.

\begin{taggedlemma}{\ref*{lem: induced}$'$} \label{lem: induced'}
Let $T$ be a measure preserving transformation of a measure space $(X, \mathcal{F}, m)$, and $A \in \mathcal{F}$ be any set recurrent for $T$ such that $0<m(A)<\infty$. Then the induced mapping $T_A$ is a  measure preserving transformation of the induced space $(A, \mathcal{F} \cap A, {m|}_A)$.
\end{taggedlemma}

Relaxing the condition $m(A)<\infty$ requires additional assumptions described in the next two statements.

\begin{lemma} \label{lem: induced}
Let $T$ be an invertible measure preserving transformation of a $\sigma$-finite measure space $(X, \mathcal{F}, m)$. Let $A \in \mathcal{F}$ be any set such that $m(A)>0$, $\{T^k A\}_{k \ge 1} \subset \mathcal{F}$, $A \subset \cup_{k \ge 1} T^k A \Mod{m}$, and  $A$ is recurrent for $T$. Then the induced mapping $T_A$ is a  measure preserving transformation of the induced space $(A, \mathcal{F} \cap A, {m|}_A)$.
\end{lemma}

\begin{remark} \label{rem: measurable inverse} 
a) If the mapping $T$ is invertible and $(X, \mathcal{F})$ is a {\it standard} measurable space, that is $X$ is a Polish space and $\mathcal F =\mathcal{B}(X)$ is its Borel $\sigma$-algebra, then the mapping $T^{-1}$ is always measurable (and hence $\{T^k A\}_{k \ge 1} \subset \mathcal F$); see~\cite[Theorem~1.0.3]{Aaronson}. In this case, the condition $A \subset \cup_{k \ge 1} T^k A \Mod{m}$ means recurrence of $A$ for the mapping $T^{-1}$, which preserves~$m$.

b) If $T^{-1}$ is measurable (and hence measure preserving) and $m(A)< \infty$, then $A$ is recurrent for $T$ if and only if $A$ is recurrent for $T^{-1}$, that is the assumptions $A \subset \cup_{k \ge 1} T^{-k} A \Mod{m}$ and $A \subset \cup_{k \ge 1} T^k A \Mod{m}$ are equivalent; see Kaimanovich~\cite[Proposition~1.3]{Kaimanovich}. 

The following example shows what can go wrong if we impose only the former of the two assumptions for a set $A$ of infinite measure: if $T$ is the shift on $\Z$ equipped with the counting measure and $A=\N$, then $T_A$ is not measure preserving on $A$ since $T_A^{-1}(\{1\}) = \varnothing$. In this example $T_A$ is not surjective$\Mod{m}$ but it is so if we require that $A \subset \cup_{k \ge 1} T^k A \Mod{m}$.

c) The assumption of $\sigma$-finiteness of $m$ can be relaxed to $\sigma$-finiteness of ${m|}_A$. The same is valid for Lemma~\ref{lem: induced2} below. Both claims can be verified easily by examining the proofs.
\end{remark}

\begin{proof}[{\bf Proofs of Lemmas~\ref{lem: induced} and~\ref{lem: induced'}}]
We need to show that $m(T_A^{-1} B) = m(B)$ for any measurable set $B \subset A$. By monotonicity of $m$, it suffices to prove this only for $B$ of finite measure since $m|_A$ is $\sigma$-finite. The rest is a standard argument (see the proof of \cite[Proposition~1.5.3]{Aaronson}), which we present for convenience of the reader.
Since the set $A$ is recurrent for $T$, we have 
\[
m(T^{-1}_A B) = \sum_{n=1}^\infty m(A \cap \{\tau_A = n \} \cap T^{-n} B) = \sum_{n=1}^\infty m(A \cap T^{-n} B \setminus \cup_{k=1}^{n-1} T^{-k} A) = \sum_{n=1}^\infty m(A \cap T^{-1} B_{n-1}),
\]
where $B_0:=B$ and $B_n := T^{-n} B \setminus \cup_{k=0}^{n-1} T^{-k} A$ for $n \ge 1$. The set $T^{-1} B_n$ of finite measure is a disjoint union of $A \cap T^{-1} B_n$ and $B_{n+1}$, hence $m( A \cap T^{-1} B_n) = m(B_n) - m(B_{n+1})$. Then
\begin{equation*} 
m(T^{-1}_A B)  = \sum_{n=1}^\infty m(A \cap T^{-1} B_{n-1}) = \sum_{n=1}^\infty \bigl( m(B_{n-1}) - m(B_n) \bigr) = m(B) - \lim_{n \to \infty} m(B_n),
\end{equation*}
and this gives (only under recurrence of the set $A$ for $T$ !) $m(T^{-1}_A B) \le m(B)$ and also
\begin{equation} \label{eq: lim B_n}
m(T^{-1}_A B) = m(B) \quad \Longleftrightarrow \quad \lim_{n \to \infty} m(B_n) =0.
\end{equation} 

Under the assumptions of Lemma~\ref{lem: induced'}, that is $m(A)<\infty$, we also have $m(T^{-1}_A (A \setminus B)) \le m(A \setminus B)$ since in the inequality $m(T^{-1}_A B) \le m(B)$ the set $B$ can be any measurable subset of the set $A$ of finite measure. Then
\[
m(T^{-1}_A A) - m(T^{-1}_A B ) = m(A) -  m(T^{-1}_A B ) \le m(A) - m(B)
\]
since $T^{-1}_A A = A \Mod{m}$ by recurrence of $A$. Thus $m(T^{-1}_A B ) \ge m(B)$, and $m(B)=m(T^{-1}_A B )$.


Under the assumptions of Lemma~\ref{lem: induced}, by invertibility of $T$ we have
\begin{equation} \label{eq: invertible vanish1}
B_n=T^{-n} B \setminus \cup_{k=0}^{n-1} T^{-k} A = T^{-n} B \setminus [T^{-n} (\cup_{k=1}^n T^k A)] = T^{-n} (B \setminus (\cup_{k=1}^n T^k A)).
\end{equation}
Hence 
\begin{equation} \label{eq: invertible vanish2}
\lim_{n \to \infty} m(B_n) = \lim_{n \to \infty}  m(B \setminus (\cup_{k=1}^n T^k A)) = m(B \setminus (\cup_{k \ge 1} T^k A)) =0,
\end{equation}
and so $m(B)=m(T^{-1}_A B )$ follows from \eqref{eq: lim B_n}.
\end{proof}

We say that the transformation $T$ is {\it ergodic} if its invariant $\sigma$-algebra $\mathcal{I}_T :=\{A \in \mathcal{F}: T^{-1} A = A \Mod{m} \} $ is $m$-trivial, i.e.\ for every $A \in \mathcal{I}_T$ either $m(A)=0$ or $m(A^c)=0$.

We say that $T$ is {\it conservative} if every measurable subset of $X$ is recurrent for $T$. By Poincar{\'e}'s recurrence theorem, $T$ is conservative if $\mu(X)<\infty$.

\begin{crit}
A measure preserving transformation $T$ of a $\sigma$-finite measure space $(X, \mathcal{F}, m)$ is conservative iff there exists a sequence of sets $\{ A_k\}_{k \ge 1} \subset \mathcal F$, all of finite measure and  recurrent for $T$, such that $X= \cup_{k \ge 1} A_k \Mod{m}$. In particular, this holds if $X = \cup_{k \ge 1} T^{-k} A \Mod{m}$, i.e.\ $\tau_A < \infty$ $m$-a.e., for some measurable set $A$ of finite measure. 
\end{crit}

The second assertion is known as {\bf Maharam's  recurrence theorem}. 

\begin{proof}
The direct implication in the first assertion is trivial. For the reverse one,  assume that there is a set $A \in \mathcal F$ of positive measure that is not recurrent for $T$. Then so is $A': = A \setminus \cup_{n=1}^\infty T^{-n} A$. Pick a  $k \ge 1$ such that $m(A_k \cap A')>0$. By Lemma~\ref{lem: induced'}, the induced mapping $T_{A_k}$ is measure preserving on the induced space $(A_k, \mathcal F  \cap A_k, {m|}_{A_k})$ of finite measure. This mapping is conservative by Poincar{\'e}'s recurrence theorem, hence $A_k \cap A'$ is a recurrent set for $T_{A_k}$, hence it is recurrent for $T$, which is a contradiction.
\end{proof}


Next we give a version of Lemma~\ref{lem: induced} for conservative transformations. The additional statement on ergodicity is in~\cite[Propositions~1.2.2 and~1.5.2]{Aaronson}. 

\begin{lemma} \label{lem: induced2}
Let $T$ be a measure preserving conservative transformation of a $\sigma$-finite measure space $(X, \mathcal{F}, m)$, and $A \in \mathcal{F}$ be any set with $m(A)>0$. Then $T_A$ is a measure preserving conservative transformation of the induced space $(A, \mathcal{F} \cap A, {m|}_A)$. Moreover, if $T$ is ergodic, then $T_A$ is ergodic and $X= \cup_{k \ge 1} T^{-k} A \Mod{m}$.
\end{lemma}

\begin{proof}
First of all, $T_A$ is well defined since $A$ is recurrent for $T$ by the conservativity. The latter property of $T$ trivially implies conservativity of $T_A$; see~\cite[Proposition~1.5.1]{Aaronson}.
Since $m$ is $\sigma$-finite, by its monotonicity it suffices to check that $m(T_A^{-1} B) = m(B)$ for any measurable set $B \subset A$ of finite positive measure. But $T_B^{-1} B = B \Mod{m}$ by conservativity of $T$, so $m(T_B^{-1} B) = m(B)$ and hence  $\lim_{n \to \infty} m(T^{-n} B \setminus \cup_{k=0}^{n-1} T^{-k} B) =0$ by \eqref{eq: lim B_n}. Since $B \subset A$, this gives $\lim_{n \to \infty} m(T^{-n} B \setminus \cup_{k=0}^{n-1} T^{-k} A) =0$, which by \eqref{eq: lim B_n} implies $m(T_A^{-1} B) = m(B)$.
\end{proof}

For invertible $T$, inducing can be reversed under additional assumption $X = \cup_{k \ge 1} T^{-k} A$ $ \Mod{m}$ using so-called {\it suspensions} (Kakutani towers). More generally, certain invariant measures of the induced transformation can be lifted to invariant measures of the original transformation, as follows. Denote $m_A:=m|_A$.

\begin{lemma} \label{lem: induced3}
Let $T$ be a measure preserving transformation of a $\sigma$-finite measure space $(X, \mathcal{F}, m)$, and let $A \in \mathcal{F}$ be any set recurrent for $T$ such that $m(A)>0$. Then for any $\sigma$-finite $T_A$-invariant measure $\nu$ on $(A, \mathcal F \cap A)$ such that $\nu \ll m_A$, the measure
\begin{equation} \label{eq: Kac}
\bar \nu(B):= \int_A \left [ \sum_{k=0}^{\tau_A(x) - 1} \I(T^k x \in B ) \right ]\! \nu(dx), \qquad B \in \mathcal{F},
\end{equation}
is invariant for $T$ and satisfies ${\bar \nu|}_A = \nu $. Moreover, if $X = \cup_{k \ge 1} T^{-k} A \Mod{m}$ and the assumptions of either Lemma~\ref{lem: induced},\ref{lem: induced'}, or~\ref{lem: induced2} are satisfied, then $\overbar{m_A} =m$. 
\end{lemma}
The assumption $\nu \ll m_A$ is imposed to ensure that $\tau_A$ is finite $\nu$-a.e.\ on $A$. In the case of  conservative $T$, equation~\eqref{eq: Kac} with $B=X$ is known as Kac's formula. 
\begin{proof}
The equality $\bar \nu|_A = \nu $ is trivial. Invariance of $\bar \nu$ is standard; see the proof of~\cite[Proposition~1.5.7]{Aaronson} or a similar argument in \eqref{eq: lifted invariant} below. 

It remains to prove the last assertion. By monotonicity and $\sigma$-finiteness of $m$, this can be checked only on sets of finite measure $m$. For any measurable set $B \subset X$,  
\begin{align} \label{eq: Kac proof 1}
\overbar{m_A}(B) =& \int_A  \left [ \sum_{n=1}^\infty \I(\tau_A(x) = n) \times \sum_{k=0}^{\tau_A(x) - 1} \I(T^k x \in B) \right ] \! m(dx) \notag \\
=& \int_A \left [ \sum_{n=1}^\infty \sum_{k=0}^{n - 1} \I(T^k x \in B, \tau_A(x) = n ) \right ] \! m(dx) \notag \\ 
=& \sum_{k=0}^\infty m (A \cap T^{-k} B \cap \{\tau_A > k\}),
\end{align}
and therefore, assuming that $m(B)<\infty$, we get
\[
\overbar{m_A}(B) = \sum_{k=0}^\infty m (A \cap T^{-k} B \setminus \cup_{n=1}^k T^{-n} A) = m(A \cap B) + \sum_{k=1}^\infty m(A \cap T^{-1} B_{k-1}'),
\]
where $B_k':= T^{-k} B \setminus \cup_{n=0}^k T^{-n} A$ for $k \ge 0$. The set $T^{-1} B_k'$ has finite measure and it is a disjoint union of $A \cap T^{-1} B_k'$ and $B_{k+1}'$, hence  $m( A \cap T^{-1} B_k') = m(B_k') - m(B_{k+1}')$. Then the sequence $m(B_k')$ is decreasing, and 
\begin{equation} \label{eq: lifted =}
\overbar{m_A}(B)= m(A \cap B) + m(B_0') - \lim_{k \to \infty} m(B_k') = m(B) -  \lim_{k \to \infty} m(B_k').
\end{equation}
It remains to show that the limit in the above formula is zero. 

Let the assumptions of Lemma~\ref{lem: induced} be satisfied. Then  the mapping $T_A$ is invertible$\Mod m$. Indeed, $T_A$ is surjective$\Mod m$ by the assumption $A \subset \cup_{k \ge 1} T^k A \Mod{m}$ of  Lemma~\ref{lem: induced}. To prove injectivity of $T_A$, assume that $T_A x_1 = T_A x_1$ for some $x_1, x_2 \in A$. Then $T^{\tau_A(x_1)} x_1= T^{\tau_A(x_2)} x_2$. W.l.o.g., assume $\tau_A(x_1) \ge \tau_A(x_2)$. Then  $T^{(\tau_A(x_1)-\tau_A(x_2))} x_1 = x_2$, which is, by definition of the first hitting time $\tau_A$ of $A$,  possible only if $\tau_A(x_1)=\tau_A(x_2)$. Hence $x_1=x_2$.
 
Furthermore, we claim that $X = \cup_{k \ge 1} T^k A \Mod{m}$. Then equality $\lim_{k \to \infty} m(B_k')=0$ follows exactly as $\lim_{n \to \infty} m(B_n)=0$ followed above from equalities \eqref{eq: invertible vanish1} and \eqref{eq: invertible vanish2}, and we get $\overbar{m_A}(B) = m(B)$ by \eqref{eq: lifted =}, as required. To prove the claim, put 
\[
V(x):=T_A^{-1}(\tilde \varphi_A(x)), \qquad x \in \{\tilde \tau_A < \infty\} \cap \tilde \varphi_A^{-1}(T_A(A)).
\]
The mapping $V$ is defined for $m$-a.e.\ $x \in X$ because of the following: $T_A$ on $A$ is invertible$\Mod m$; $\tilde \tau_A$ is finite$\Mod m$ since $X = \cup_{k \ge 1} T^{-k} A \Mod{m}$ by  assumption of Lemma~\ref{lem: induced3}; and 
\[
m\big( \{\tilde \tau_A < \infty\} \setminus \tilde \varphi_A^{-1}(T_A(A)) \big) = m \big( \tilde \varphi_A^{-1}(A \setminus T_A(A)) \big) \le m \big(  \cup_{k=0}^\infty T^{-k} (A \setminus T_A(A)) \big) =0.
\]
Then for $m$-a.e.\ $x$, we have $T^{k(x)} (V(x))=x$ for a positive integer $k(x)$ satisfying $k(x)=\tau_A(V(x))-\tilde \tau_A(x) $. Thus, equality $X = \cup_{k \ge 1} T^k A \Mod{m}$ holds, as claimed.

Let now the assumptions of either Lemma~\ref{lem: induced'} or~\ref{lem: induced2} be satisfied. Then the transformation $T$ is conservative. In the latter case this is by assumption. In the former case, this follows by Maharam's recurrence theorem since $m(A)<\infty$ and $X = \cup_{k \ge 1} T^{-k} A \Mod{m}$ by  assumption of Lemma~\ref{lem: induced3}. When $m(A)<\infty$, the equality $\overbar{m_A} = m$ holds by \cite[Lemma~1.5.4]{Aaronson}. The following argument covers both cases of finite and infinite $m(A)$. 

For any integer $N \ge 1$, denote $B^{(N)}:=B \cap (\cup_{n=1}^N T^{-n} A)$. Notice that for any $k \ge N$, we have $\{ \tau_{B^{(N)}} \le k-N \} \subset \{ \tau_A \le k \}$, hence
\[
T^{-k} (B^{(N)}) \setminus \cup_{n=0}^k T^{-n} A \subset \{ k-N <\tau_{B^{(N)}} \le k\}, \qquad k \ge N.
\]
Then for $k \ge N$,
\begin{align*}
m(B_k')&=m \bigl(T^{-k} (B \setminus \cup_{n=1}^N T^{-n} A) \setminus \cup_{n=1}^k T^{-n} A \bigr) +m\bigl(T^{-k} (B^{(N)}) \setminus \cup_{n=1}^k T^{-n} A \bigr) \\
&\le m \bigl(T^{-k} (B \setminus \cup_{n=1}^N T^{-n} A) \bigr) + N \sup_{n > k-N} m(\tau_{B^{(N)}} = n) \\
&= m (B \setminus \cup_{n=1}^N T^{-n} A ) + N \sup_{n > k-N} m (T^{-n}(B^{(N)}) \setminus \cup_{i=1}^{n-1} T^{-i} B^{(N)}).
\end{align*}
The first term in the last line can be made as small as necessary by choosing $N$ to be large enough, and the second term vanishes as $k \to \infty$ for any fixed $N$ by equivalence \eqref{eq: lim B_n} applied with $B^{(N)}$ substituted for $A$ and $B$. Such application of \eqref{eq: lim B_n} is possible since  the set $B^{(N)}$ is recurrent for $T$ because $T$ is conservative, as explained above, and $m(B^{(N)}) \le m(B) < \infty$. Thus, $\lim_{k \to \infty} m(B_k')=0$, and by \eqref{eq: lifted =}, this yields the required equality $\overbar{m_A}(B)= m(B)$.
\end{proof}

Finally, we recall the following classical result; see Zweim\"uller~\cite{Zweimuller}.
\begin{Hopf}
Let $T$ be a conservative ergodic measure preserving transformation of a $\sigma$-finite measure space $(X, \mathcal{F}, m)$. Then for any functions $f, g \in L^1(X, \mathcal{F}, m)$ with non-zero $g \ge 0$, 
\begin{equation*} 
\lim_{n \to \infty} \frac{\sum_{k=0}^{n-1} f \circ T^k}{\sum_{k=0}^{n-1} g \circ T^k}  = \frac{\int_X f d m}{\int_X g d m}, \quad m\text{-a.e.} 
\end{equation*}
\end{Hopf}

\bibliographystyle{plain}
\bibliography{overshoot}

\end{document}